\documentclass[english]{article}
\usepackage[T1]{fontenc}
\usepackage[latin9]{inputenc}
\usepackage{amsmath}
\usepackage{amsthm}
\usepackage{amssymb}
\makeatletter

\providecommand{\tabularnewline}{\\}

\numberwithin{equation}{section}
\numberwithin{figure}{section}

\usepackage{amsthm}\usepackage{fullpage}\usepackage{makeidx}\usepackage{amsfonts}\usepackage{latexsym}\makeindex

\def\Z{{\mathbb{Z}}}

\theoremstyle{definition}
\newtheorem{lemma}{Lemma}[section]
\newtheorem{theorem}[lemma]{Theorem}

\newtheorem{proposition}[lemma]{Proposition}
\newtheorem{definition}[lemma]{Definition}

\newtheorem{remark}[lemma]{Remark}
\def \<{\langle}
\def \>{\rangle}
\def\wt{\rm wt}

\usepackage{times}

\usepackage{babel}

\usepackage{babel}

\usepackage[blocks]{authblk}
\title{The $3$-permutation orbifold of a lattice vertex operator algebra
}

\author{Chongying Dong\footnote{Supported by NSF grant DMS-1404741 and China NSF grant 11371261}}
\affil{School of Mathematics and Statistics, Qingdao University, Qingdao 266071 CHINA \& Department of Mathematics, University of
California, Santa Cruz, CA 95064 USA}
\author{Feng Xu\footnote{Partially supported by China NSF grant 11471064}}

\affil{Department of Mathematics, Foshan University, Foshan 528000 CHINA \&  Department of Mathematics, University of California, Riverside, CA 92521 USA}
\author{Nina Yu\footnote{Supported by China NSF grant 11601452}}
\affil{School of Mathematical Sciences, Xiamen University, Fujian 361005 CHINA}

\usepackage{babel}

\usepackage{babel}

\usepackage{babel}

\makeatother

\usepackage{babel}
\begin{document}
\maketitle
\begin{abstract}
Irreducible modules of the 3-permutation orbifold of a rank
one lattice vertex operator algebra are listed explicitly. Fusion rules are
determined by using the quantum dimensions. The $S$-matrix is also given.
\end{abstract}

\section{{\normalsize{}Introduction}}

This paper is a continuation of our study on permutation orbifolds of
lattice vertex operator algebras \cite{DXY1,DXY2}. Let $V$ be a
vertex operator algebra, then the tensor product $V^{\otimes n}$
of $n$-copies of $V$ naturally has
a vertex operator algebra structure \cite{FHL}. Any element $\sigma$
of the symmetric group $S_{n}$ acts on $V^{\otimes n}$ which gives
an automorphism of $V^{\otimes n}$ of finite order. The fixed points
set $\left(V^{\otimes n}\right)^{\<\sigma\>}=\left\{ v\in V^{\otimes n}|\sigma\left(v\right)=v\right\} $
is a vertex operator subalgebra which is called a $\sigma$-permutation
orbifold model.

The permutation orbifold theory has been studied by physicists two
decades ago. Its systematic study has been started with the paper
\cite{BHS} for cyclic permutation for affine algebras and the Virasoro
algebras, and was generalized to the general symmetric group in \cite{Ba}.
Conformal nets approach to permutation orbifold have been given
in \cite{KLX} where irreducible representations of the cyclic orbifold
are determined and fusion rules were given for $n=2$. Twisted sectors
of permutation orbifolds of tensor products of an arbitrary vertex
operator algebra have been constructed in \cite{BDM}. The $C_{2}$-cofiniteness
of permutation orbifolds and general cyclic orbifolds have been established
in \cite{A1,A2,M1,M2}. An equivalence of two constructions \cite{FLM,Le, BDM} of twisted modules for permutation orbifolds of
lattice vertex operator algebras was given in \cite{BHL}.
 The permutation orbifolds with $n=2$ has been studied extensively. In particular, the quantum dimensions of the irreducible modules and fusion
rules for the permutation orbifolds of lattice vertex
operator algebras are determined \cite{DXY1,DXY2}. Also see \cite{BE} for the classification of irreducible modules for the fixed point vertex operator subalgebra of a lattice vertex operator algebra under any order isometry of the lattice.

The main ingredient in orbifold theory is the twisted modules. Since the twisted modules for permutation orbifolds have been constructed in \cite{BDM}, a complete understanding of general permutation orbifolds is visible.
According to recent results in \cite{M2} and \cite{CM}, the cyclic permutation orbifold $(V^{\otimes n})^{\<\sigma\>}$ is rational and $C_2$-cofinite if $V$ is.   So every irreducible $(V^{\otimes n})^{\<\sigma\>}$-module occurs in an irreducible twisted module \cite{DRX}. The question now for the permutation orbifolds is how to write down the irreducible modules explicitly, and compute the quantum dimensions and the fusion rules. But it is still impossible to compute the fusion rules for an arbitrary rational vertex operator algebra $V$ at this stage.
In this paper we achieve this for $3$-permutation orbifold of a rank one lattice vertex operator algebra. We also find the $S$-matrix. In this case, $(V^{\otimes 3})^{\<\sigma\>}$ is a simple current extension of a rational vertex operator algebra. It follows from \cite{Y} directly that $(V^{\otimes 3})^{\<\sigma\>}$ is rational. Our results depend heavily on those  given  in \cite{TY1,TY2} on orbifold vertex operator algebra $V_{\sqrt{2}A_2}^{\tau}$ where  $\tau$ is an order $3$ isometry of $\sqrt{2}A_2.$ This is why we only consider this particular $3$-orbifold. Our result also gives an answer to a question posed at the end of \cite{KLX}. The question is to determine the chiral data such as $S,T$ matrices for $n$-cyclic orbifolds when $n>2$. As for $n=2$ case computed in the last section of \cite{KLX}, it is expected that a general result for $n>2$ case will lead to interesting arithmetic properties of chiral data.  Even though this question  at the end of \cite{KLX} is stated in the setting of conformal net, such question is also interesting in the framework of vertex operator algebra.

The paper is organized as follows: In Section 2 we introduce basic
notions in the theory of vertex operator algebras. In Section 3, we
first recall construction of lattice vertex operator algebras and
results about the fixed point subalgebra $V_{\sqrt{2}A_{2}}^{\tau}$
of lattice vertex operator algebra $V_{\sqrt{2}A_{2}}$ where $\tau$
is an order 3 automorphism of $V_{\sqrt{2}A_{2}}$. We list the
irreducible modules for $\left(V_{\mathbb{Z}\alpha}\otimes V_{\mathbb{Z}\alpha}\otimes V_{\mathbb{Z}\alpha}\right)^{\mathbb{Z}_{3}}$ explicitly.
 In Section 4 we determine the  quantum dimensions of irreducible
$\left(V_{\mathbb{Z}\alpha}\otimes V_{\mathbb{Z}\alpha}\otimes V_{\mathbb{Z}\alpha}\right)^{\mathbb{Z}_{3}}$-modules
and the fusion products. In the last section, we give the
$S$-matrix of $\left(V_{\mathbb{Z}\alpha}\otimes V_{\mathbb{Z}\alpha}\otimes V_{\mathbb{Z}\alpha}\right)^{\mathbb{Z}_{3}}$.

\section{{\normalsize{}Preliminaries}}

Let $\left(V,Y,\mathbf{1},\omega\right)$ be a vertex operator algebra.
First we review basics from \cite{FLM, FHL, DLM2, LL}.  Let $g$ be an automorphism of a vertex operator algebra $V$ of order
$T$. Denote the decomposition of $V$ into eigenspaces of $g$ as:

\[
V=\oplus_{r\in\mathbb{Z}/T\text{\ensuremath{\mathbb{Z}}}}V^{r}
\]
where $V^{r}=\left\{ v\in V|gv=e^{2\pi ir/T}v\right\} $.

\begin{definition}A \emph{weak $g$-twisted $V$-module} $M$ is
a vector space with a linear map
\[
Y_{M}:V\to\left(\text{End}M\right)\{z\}
\]

\[
v\mapsto Y_{M}\left(v,z\right)=\sum_{n\in\mathbb{Q}}v_{n}z^{-n-1}\ \left(v_{n}\in\mbox{End}M\right)
\]

which satisfies the following: for all $0\le r\le T-1$, $u\in V^{r}$,
$v\in V$, $w\in M$,

\[
Y_{M}\left(u,z\right)=\sum_{n\in-\frac{r}{T}+\mathbb{Z}}u_{n}z^{-n-1},
\]

\[
u_{l}w=0\ for\ l\gg0,
\]

\[
Y_{M}\left(\mathbf{1},z\right)=Id_{M},
\]

\[
z_{0}^{-1}\text{\ensuremath{\delta}}\left(\frac{z_{1}-z_{2}}{z_{0}}\right)Y_{M}\left(u,z_{1}\right)Y_{M}\left(v,z_{2}\right)-z_{0}^{-1}\delta\left(\frac{z_{2}-z_{1}}{-z_{0}}\right)Y_{M}\left(v,z_{2}\right)Y_{M}\left(u,z_{1}\right)
\]

\[
=z_{2}^{-1}\left(\frac{z_{1}-z_{0}}{z_{2}}\right)^{-r/T}\delta\left(\frac{z_{1}-z_{0}}{z_{2}}\right)Y_{M}\left(Y\left(u,z_{0}\right)v,z_{2}\right),
\]
where $\delta\left(z\right)=\sum_{n\in\mathbb{Z}}z^{n}$. \end{definition}

\begin{definition}An \emph{admissible $g$-twisted $V$-module} $M=\oplus_{n\in\frac{1}{T}\mathbb{Z}_{+}}M\left(n\right)$
is a $\frac{1}{T}\mathbb{Z}_{+}$-graded weak $g$-twisted module
such that $u_{m}M\left(n\right)\subset M\left(\mbox{wt}u-m-1+n\right)$
for homogeneous $u\in V$ and $m,n\in\frac{1}{T}\mathbb{Z}.$ $ $

\end{definition}

\begin{definition}

A $g$-\emph{twisted $V$-module} is a weak $g$-twisted $V$-module
$M$ which carries a $\mathbb{C}$-grading induced by the spectrum
of $L(0)$ where $L(0)$ is the component operator of $Y(\omega,z)=\sum_{n\in\mathbb{Z}}L(n)z^{-n-2}.$
That is, we have $M=\bigoplus_{\lambda\in\mathbb{C}}M_{\lambda},$
where $M_{\lambda}=\{w\in M|L(0)w=\lambda w\}$. Moreover, $\dim M_{\lambda}$
is finite and for fixed $\lambda,$ $M_{\frac{n}{T}+\lambda}=0$ for
all small enough integers $n.$

\end{definition}

If $g=Id_{V}$ we have the notions of weak, ordinary and admissible
$V$-modules \cite{DLM2}.

\begin{remark} If $M$ is a simple  $g$-twisted $V$-module, then there exists $\lambda\in\mathbb{C}$ such that  $M=\oplus_{n=0}^{\infty}M_{\lambda+\frac{n}{T}}$ with $M_\lambda\not=0$ and $M_{\lambda+\frac{n}{T}}=\{w\in M|L(0)w=(\lambda+\frac{n}{T})w\}$ for any $n\in \mathbb{Z}_+$. $\lambda$ is called the \emph{ conformal weight} of $M$.
\end{remark}

\begin{definition}A vertex operator algebra $V$ is called \emph{$g$-rational}
if the admissible $g$-twisted module category is semisimple. $V$
is called \emph{rational} if $V$ is $1$-rational. \end{definition}

It is proved in \cite{DLM2} that if $V$ is a $g$-rational vertex
operator algebra, then there are only finitely many irreducible admissible
$g$-twisted $V$-modules up to isomorphism and any irreducible admissible
$g$-twisted $V$-module is ordinary.

\begin{definition} A vertex operator algebra $V$ is said to be $C_{2}$-cofinite
if $V/C_{2}(V)$ is finite dimensional, where $C_{2}(V)=\langle v_{-2}u|v,u\in V\rangle$ \cite{Z}.
\end{definition}

Now we consider the tensor product vertex operator algebras and the tensor
product modules for tensor product vertex operator algebras. The tensor
product of vertex operator algebras $\left(V^{1},Y^{1},1,\omega^{1}\right)$,
$\left(V^{2},Y^{2},1,\omega^{2}\right)$ and $\left(V^{3},Y^{3},1,\omega^{3}\right)$
is constructed on the tensor product vector space $V=V^{1}\otimes V^{2}\otimes V^{3}$
where the vertex operator $Y\left(\cdot,z\right)$ is defined by $Y\left(v^{1}\otimes v^{2}\otimes v^{3},z\right)=Y\left(v^{1},z\right)\otimes Y\left(v^{2},z\right)\otimes Y\left(v^{3},z\right)$
for $v^{i}\in V^{i}$, $i=1,2,3$, the vacuum vector is $\mathbf{1}=1\otimes1\otimes1$
and the Virasoro element is $\omega=\omega^{1}\otimes\omega^{2}\otimes\omega^{3}$.
Then $\left(V,Y,\mathbf{1},\omega\right)$ is a vertex operator algebra
\cite{FHL,LL}. Let $W^{i}$ be an admissible $V^{i}$-module for
$i=1,2,3$. We may construct the tensor product admissible module
$W^{1}\otimes W^{2}\otimes W^{3}$ for the tensor product vertex operator
algebra $V^{1}\otimes V^{2}\otimes V^{3}$ by $Y\left(v^{1}\otimes v^{2}\otimes v^{3},z\right)=Y\left(v^{1},z\right)\otimes Y\left(v^{2},z\right)\otimes Y\left(v^{3},z\right)$.
Then $W^{1}\otimes W^{2}\otimes W^{3}$ is an admissible $V^{1}\otimes V^{2}\otimes V^{3}$-module.
We have the following result about tensor product modules \cite{DMZ,FHL}:

\begin{theorem} Let $V^{1},V^{2},V^{3}$ be rational vertex operator
algebras, then $V^{1}\otimes V^{2}\otimes V^{3}$ is rational and
any irreducible $V^{1}\otimes V^{2}\otimes V^{3}$-module is a tensor
product $W^{1}\otimes W^{2}\otimes W^{3}$ for some irreducible $V^{i}$-module
$W^{i}$ and $i=1,2,3.$

\end{theorem}

Now we recall notion of contragredient module \cite{FHL}. \begin{definition}
Let $M=\bigoplus_{n\in\frac{1}{T}\mathbb{Z}_{+}}M(n)$ be an admissible
$g$-twisted $V$-module, the contragredient module $M'$ is defined
as follows:
\[
M'=\bigoplus_{n\in\frac{1}{T}\mathbb{Z}_{+}}M(n)^{*},
\]
where $M(n)^{*}=\mbox{Hom}_{\mathbb{C}}(M(n),\mathbb{C}).$ The vertex
operator $Y_{M'}(v,z)$ is defined for $v\in V$ via
\begin{eqnarray*}
\langle Y_{M'}(v,z)f,u\rangle= & \langle f,Y_{M}(e^{zL(1)}(-z^{-2})^{L(0)}v,z^{-1})u\rangle
\end{eqnarray*}
where $\langle f,w\rangle=f(w)$ is the natural paring $M'\times M\to\mathbb{C}.$
Then $M'$ is an admissible $g^{-1}$-twisted $V$-module \cite{X}.
A $V$-module $M$ is said to be \emph{ self dual} if $M$ and $M'$
are isomorphic $V$-modules. \end{definition}

Here are the definition of intertwining operators and fusion rules
\cite{FHL}.

\begin{definition} Let $(V,\ Y)$ be a vertex operator algebra and
let $(W^{1},\ Y^{1}),\ (W^{2},\ Y^{2})$ and $(W^{3},\ Y^{3})$ be
$V$-modules. An intertwining operator of type $\left(\begin{array}{c}
W^{1}\\
W^{2\ }W^{3}
\end{array}\right)$ is a linear map
\[
I(\cdot,\ z):\ W^{2}\to\text{\ensuremath{\mbox{Hom}(W^{3},\ W^{1})\{z\}}}
\]

\[
u\to I(u,\ z)=\sum_{n\in\mathbb{Q}}u_{n}z^{-n-1}
\]
satisfying:

(1) for any $u\in W^{2}$ and $v\in W^{3}$, $u_{n}v=0$ for $n$
sufficiently large;

(2) $I(L_{-1}v,\ z)=(\frac{d}{dz})I(v,\ z)$;

(3) (Jacobi Identity) for any $u\in V,\ v\in W^{2}$

\[
z_{0}^{-1}\delta\left(\frac{z_{1}-z_{2}}{z_{0}}\right)Y^{1}(u,\ z_{1})I(v,\ z_{2})-z_{0}^{-1}\delta\left(\frac{-z_{2}+z_{1}}{z_{0}}\right)I(v,\ z_{2})Y^{3}(u,\ z_{1})
\]
\[
=z_{2}^{-1}\left(\frac{z_{1}-z_{0}}{z_{2}}\right)I(Y^{2}(u,\ z_{0})v,\ z_{2}).
\]

We denote the space of all intertwining operators of type $\left(\begin{array}{c}
W^{1}\\
W^{2}\ W^{3}
\end{array}\right)$ by $I_{V}\left(\begin{array}{c}
W^{1}\\
W^{2}\ W^{3}
\end{array}\right).$ Let $N_{W^{2},\ W^{3}}^{W^{1}}=\dim I_{V}\left(\begin{array}{c}
W^{1}\\
W^{2}\ W^{3}
\end{array}\right)$. These integers $N_{W^{2},\ W^{3}}^{W^{1}}$ are usually called the
\emph{fusion rules}. \end{definition}

\begin{definition} Let $V$ be a vertex operator algebra, and $W^{1},$
$W^{2}$ be two $V$-modules. A pair $(W,I)$ of a $V$-module $W$
and $I\in I_{V}\left(\begin{array}{c}
\ \ W\ \\
W^{1}\ \ W^{2}
\end{array}\right),$ is called a \emph{fusion product} of $W^{1}$ and $W^{2}$ if for
any $V$-module $M$ and $\mathcal{Y}\in I_{V}\left(\begin{array}{c}
\ \ M\ \\
W^{1}\ \ W^{2}
\end{array}\right),$ there is a unique $V$-module homomorphism $f:W\rightarrow M,$ such
that $\mathcal{Y}=f\circ I.$ As usual, we denote $W$ by $W^{1}\boxtimes_{V}W^{2}.$
\end{definition}

It is well known that if $V$ is rational, then the fusion product
exists. We shall often consider the fusion product
\[
W^{1}\boxtimes_{V}W^{2}=\sum_{W}N_{W^{1},\ W^{2}}^{W}W
\]
where $W$ runs over the set of equivalence classes of irreducible
$V$-modules.

The following symmetry properties of fusion rules are well known \cite{FHL}.

\begin{proposition}\label{fusion rule symmmetry property} Let $W^{i}$
$\left(i=1,2,3\right)$ be $V$-modules. Then

\[
\dim I_{V}\left(_{W^{1}W^{2}}^{\ \ W^{3}}\right)=\dim I_{V}\left(_{W^{2}W^{1}}^{\ \ W^{3}}\right),\dim I_{V}\left(_{W^{1}W^{2}}^{\ \ W^{3}}\right)=\dim I_{V}\left(_{W^{1}\left(\ W^{3}\right)'}^{\ \ \left(W^{2}\right)'}\right).
\]

\end{proposition}

\begin{definition} Let $V$ be a simple vertex operator algebra.
An irreducible $V$-module $M$ is called \emph{a simple current}
if for any irreducible $V$-module $W$, $W\boxtimes M$ exists and
is also a simple $V$-module. \end{definition}

\begin{definition} Let $g$ be an automorphism of the vertex operator
algebra $V$ with order $T$. Let $M=\oplus_{n\in\frac{1}{T}\mathbb{Z}_{+}}M_{\lambda+n}$
be a $g$-twisted $V$-module, the formal character of $M$ is defined
as

\[
\mbox{ch}_{q}M=\mbox{tr}_{M}q^{L\left(0\right)-c/24}=q^{\lambda-c/24}\sum_{n\in\frac{1}{T}\mathbb{Z}_{+}}\left(\dim M_{\lambda+n}\right)q^{n},
\]
where $\lambda$ is the conformal weight of $M$. \end{definition}

If $V$ is $C_{2}$-cofinite, then $\mbox{ch}_{q}M$ converges to
a holomorphic function on the domain $\left|q\right|<1$ \cite{Z,DLM3}.
We denote the holomorphic function by $Z_{M}\left(\tau\right)$.
Here and below, $\tau$ is in the upper half plane $\mathbb{H}$
and $q=e^{2\pi i\tau}$.

\begin{definition} \label{quantum dimension}Let $V$ be a vertex
operator algebra and $M$ a $g$-twisted $V$-module such that $Z_{V}\left(\tau\right)$
and $Z_{M}\left(\tau\right)$ exists. The quantum dimension of $M$
over $V$ is defined as
\[
\mbox{qdim}_{V}M=\lim_{y\to0}\frac{Z_{M}\left(iy\right)}{Z_{V}\left(iy\right)},
\]
where $y$ is real and positive. \end{definition}

 Assume
$V$ is a rational, $C_{2}$-cofinite vertex operator algebra of CFT
type with $V\cong V'$. Let $M^{0}\cong V,\,M^{1},\,\cdots,\,M^{d}$
denote all inequivalent irreducible $V$-modules. Also assume the
conformal weights $\lambda_{i}$ of $M^{i}$ are positive for all
$i>0.$ Then quantum dimensions exist \cite{DJX}. Moreover, we have the following properties of quantum dimensions
\cite{DJX}  which can be derived from the Verlinde formula \cite{H1}, \cite{H2}.

\begin{proposition}\label{possible values of quantum dimensions}
(1) $q\dim_{V}M^{i}\geq1,$ $\forall i=0,\cdots,d.$

(2) For any $i,\,j=0,\cdots,\,d,$
\[
q\dim_{V}\left(M^{i}\boxtimes M^{j}\right)=q\dim_{V}M^{i}\cdot q\dim_{V}M^{j}.
\]

(3) $M^i$ is a simple current if and only if $q\dim_{V}M^i=1$.
\end{proposition}

\section{{\normalsize{}The vertex operator algebra $\left(V_{\Z\alpha}\otimes V_{\Z\alpha}\otimes V_{\Z\alpha}\right)^{\mathbb{Z}_{3}}$}}

We first review some facts about lattice vertex operator algebra $V_{\mathcal{L}}$
associate to a positive even lattice $\mathcal{L}$ from \cite{FLM}.
Then we give some related results about $V_{\mathcal{L}}$ \cite{D, DL}.

Let $\mathcal{L}$ be a positive definite even lattice with bilinear
form $\left\langle \cdot,\cdot\right\rangle $ and $\mathcal{L}^{\circ}$
its dual lattice in $\mathfrak{h}=\mathbb{C}\otimes_{\mathbb{Z}}\mathcal{L}$.
Let $\left\{ \lambda_{0}=0,\lambda_{1},\lambda_{2},\cdots\right\} $
be a complete set of coset representatives of $\mathcal{L}$ in $\mathcal{L}^{\circ}.$
Then the lattice vertex operator algebra $V_{\mathcal{L}}$ is rational
and $V_{\lambda_{i}+\mathcal{L}}$ are the irreducible $V_{\mathcal{L}}$-modules
\cite{Bo,FLM,D,DLM1}.

The fusion rules for irreducible $V_{\mathcal{L}}$-modules are given
by the following \cite{DL}:

\begin{proposition}\label{Fusion-V_L} $N_{_{V_{\mathcal{L}}}}\left(_{V_{\lambda+\mathcal{L}}\ V_{\mu+\mathcal{L}}}^{V_{\gamma+\mathcal{L}}}\right)=\delta_{\lambda+\mu+\mathcal{L},\gamma+\mathcal{L}}$
for $\lambda,\mu$ and $\gamma\in\mathcal{L}^{\circ}$.\end{proposition}

Consider the lattice vertex operator algebra $V_{A_2}$ associated to the root lattice $A_2.$
Let $\gamma_{1}$, $\gamma_{2}$ be the simple roots. Set $\gamma_{0}=-\left(\gamma_{1}+\gamma_{2}\right)$. Then $\left\langle \gamma_{i},\gamma_{i}\right\rangle =2$
and $\left\langle \gamma_{i},\gamma_{j}\right\rangle =-1$ if $i\not=j$.
Let $\tau$ be an isometry of $A_{2}$ defined by
\[
\gamma_{1}\mapsto\gamma_{2}\mapsto\gamma_{0}\mapsto\gamma_{1}.
\]
 Then $\tau$ is a fixed point free isometry of order 3 that can be
lifted naturally to an automorphism of $V_{A_{2}}$.

Let $\mathcal{L}=\mathbb{Z}\alpha\oplus\mathbb{Z}\alpha\oplus\mathbb{Z}\alpha$
where $\left\langle \alpha,\alpha\right\rangle =2$. Set $\alpha^{1}=\left(\alpha,0,0\right)$,
$\alpha^{2}=\left(0,\alpha,0\right)$ and $\alpha^{3}=\left(0,0,\alpha\right)\in\mathcal{L}$.
Let
\[
\beta_{1}=\alpha^{1}+\alpha^{2}+\alpha^{3},\beta_{2}=\alpha^{1}-\alpha^{2},\beta_{3}=\alpha^{2}-\alpha^{3},
\]
 then we have $\left\langle \beta_{1},\beta_{1}\right\rangle =6$,
$\left\langle \beta_{2},\beta_{2}\right\rangle =\left\langle \beta_{3},\beta_{3}\right\rangle =4$,
$\left\langle \beta_{1},\beta_{2}\right\rangle =\left\langle \beta_{1},\beta_{3}\right\rangle =0$,
and $\left\langle \beta_{2},\beta_{3}\right\rangle =-2$.

Let $L=\mathbb{Z}\beta_{2}+\mathbb{Z}\beta_{3}$ be the lattice spanned
by $\beta_{2}$ and $\beta_{3}$. Then $L$ is isometric to $\sqrt{2}A_{2}$ and regard $\tau$ as an isometry of $L$ in an obvious way.
Set $\beta_{0}=-\left(\beta_{2}+\beta_{3}\right)$, then the cosets
of $L$ in the dual lattice $L^{\circ}=\left\{ \beta\in\mathbb{Q}\otimes_{\mathbb{Z}}L|\left(\beta,L\right)\subset\mathbb{Z}\right\} $
are:

\[
L^{0}=L,\ L^{1}=\frac{2\beta_{2}+\beta_{3}}{3}+L,\ L^{2}=\frac{\beta_{2}+2\beta_{3}}{3}+L,
\]

\[
L_{0}=L,\ L_{a}=\frac{\beta_{3}}{2}+L,\ L_{b}=\frac{\beta_{0}}{2}+L,\ L_{c}=\frac{\beta_{2}}{2}+L.
\]
Set $L^{\left(i,j\right)}=L_{i}+L^{j},$ for $i=0,a,b,c$ and $j=0,1,2$.

\begin{remark} Note that $\tau$ can be lifted to an automorphism of $V_L$ and we also denote this automorphism by $\tau.$ For a simple $V_{L}$-module $\left(U,Y_{U}\right)$,
let $\left(U\circ\tau,Y_{U\circ\tau}\right)$ be a new $V_{L}$-module
where $U\circ\tau=U$ as vector spaces and $Y_{U\circ\tau}\left(v,z\right)=Y_{U}\left(\tau v,z\right)$
for $v\in V_{L}$ \cite{DLM3}. If $U$ and $U\circ\tau$ are equivalent
$V_{L}$-modules, $U$ is said to be $\tau$-stable. The following
result is from  \cite{TY1}.
\end{remark}

\begin{lemma} \label{tau-stable}(1) $V_{L^{\left(0,j\right)}},j=0,1,2$
are $\tau$-stable.

(2) $V_{L^{\left(a,j\right)}}\circ\tau=V_{L^{\left(c,j\right)}}$,
$V_{L^{\left(c,j\right)}}\circ\tau=V_{L^{\left(b,j\right)}},$ and
$V_{L^{\left(b,j\right)}}\circ\tau=V_{L^{\left(a,j\right)}},j=0,1,2.$

(3) $V_L^{\tau}$ is rational and $C_2$-cofinite.
\end{lemma}

Note that $\alpha^{1}=\frac{\beta_{1}+2\beta_{2}+\beta_{3}}{3}$,
$\alpha^{2}=\frac{\beta_{1}-\beta_{2}+\beta_{3}}{3}$, $\alpha^{3}=\frac{\beta_{1}-\beta_{2}-2\beta_{3}}{3}.$
We have
\begin{equation}
\begin{aligned}\mathcal{L} & =\mathbb{Z}\alpha\oplus\mathbb{Z}\alpha\oplus\mathbb{Z}\alpha\\
 & =\mathbb{Z}\frac{\beta_{1}+2\beta_{2}+\beta_{3}}{3}\oplus\mathbb{Z}\frac{\beta_{1}-\beta_{2}+\beta_{3}}{3}\oplus\mathbb{Z}\frac{\beta_{1}-\beta_{2}-2\beta_{3}}{3}\\
 & =\left(\mathbb{Z}\beta_{1}\oplus L\right)\cup\left(\left(\frac{1}{3}\beta_{1}+\mathbb{Z}\beta_{1}\right)\oplus\left(\frac{2\beta_{2}+\beta_{3}}{3}+L\right)\right)\\
 & \cup\left(\left(\frac{2}{3}\beta_{1}+\mathbb{Z}\beta_{1}\right)\oplus\left(\frac{\beta_{2}+2\beta_{3}}{3}+L\right)\right).
\end{aligned}
\label{decomposition of lattice}
\end{equation}
Thus
\[
V_{\mathbb{Z}\alpha}\otimes V_{\mathbb{Z}\alpha}\otimes V_{\mathbb{Z}\alpha}\cong V_{\mathbb{Z}\beta_{1}}\otimes V_{L}+V_{\frac{1}{3}\beta_{1}+\mathbb{Z}\beta_{1}}\otimes V_{\frac{2\beta_{2}+\beta_{3}}{3}+L}+V_{\frac{2}{3}\beta_{1}+\mathbb{Z}\beta_{1}}\otimes V_{\frac{\beta_{2}+2\beta_{3}}{3}+L}
\]
 where $V_{\frac{1}{3}\beta_{1}+\mathbb{Z}\beta_{1}}$, $V_{\frac{2}{3}\beta_{1}+\mathbb{Z}\beta_{1}}$
are irreducible modules of $V_{\mathbb{Z}\beta_{1}}$ and $V_{\frac{2\beta_{2}+\beta_{3}}{3}+L}$
, $V_{\frac{\beta_{2}+2\beta_{3}}{3}+L}$ are irreducible modules
of $V_{L}.$

Let $\sigma=\left(1\ 2\ 3\right)$ be the isometry of $\mathcal{L}$ permuting $\alpha_1,\alpha_2, \alpha_3.$ We also use  $\sigma$ to denote the corresponding permutation automorphism of $V_{\mathbb{Z}\alpha}\otimes V_{\mathbb{Z}\alpha}\otimes V_{\mathbb{Z}\alpha}$. Clearly, the group $\<\sigma\>$ generated by $\sigma$ is isomorphic to $\Z_3.$
Note that $\sigma\left(\beta_{1}\right)=\beta_{1}$, $\sigma\left(\beta_{2}\right)=\beta_{3}$,
$\sigma\left(\beta_{3}\right)=-\left(\beta_{2}+\beta_{3}\right)$.
So $\sigma|_{V_{L}}=\tau$ and hence

\begin{alignat*}{1}
\left(V_{\mathbb{Z}\alpha}\otimes V_{\mathbb{Z}\alpha}\otimes V_{\mathbb{Z}\alpha}\right)^{\mathbb{Z}_{3}} & \cong V_{\mathbb{Z}\beta_{1}}\otimes V_{L}^{\tau}+V_{\frac{1}{3}\beta_{1}+\mathbb{Z}\beta_{1}}\otimes V_{\frac{2\beta_{2}+\beta_{3}}{3}+L}^{\tau}+V_{\frac{2}{3}\beta_{1}+\mathbb{Z}\beta_{1}}\otimes V_{\frac{\beta_{2}+2\beta_{3}}{3}+L}^{\tau}
\end{alignat*}
For short, we denote $\left(V_{\mathbb{Z}\alpha}\otimes V_{\mathbb{Z}\alpha}\otimes V_{\mathbb{Z}\alpha}\right)^{\mathbb{Z}_{3}}$  by $\mathcal{U}$.

The orbifold vertex operator algebra $V_{L}^{\tau}$ has been well
studied in \cite{TY1,TY2,C,CL}. For any $\tau$-invariant $V_{L}$-module
$U$ and $\epsilon\in\mathbb{Z}_{3}$, denote
\[
U\left[\epsilon\right]=\left\{ u\in U|\tau u=e^{\frac{2\pi i\epsilon}{3}}u\right\} .
\]

Recall from \cite{TY1} that for $k=1,2$, $V_L$ has three inequivalent irreducible $\tau^k$-twisted modules $V_{L}^{T,j}(\tau^{k})$ for $j\in\Z_3$ and $\tau$ acts on  $V_{L}^{T,j}(\tau^{k})$ with
eigenspaces $V_{L}^{T,j}\left(\tau^{k}\right)\left[\epsilon\right]$ for $\epsilon\in\mathbb{Z}_{3}.$
The irreducible modules for $V_{L}^{\tau}$ are classified in \cite{TY1}:
\begin{proposition}\label{rationality of V_Ltau}The vertex operator
algebra $V_{L}^{\tau}$ is a simple, rational, $C_{2}$-cofinite,
and of CFT type. There are exactly 30 irreducible $V_{L}^{\tau}$-modules
up to isomorphism:

(1) $V_{L^{\left(0,j\right)}}\left[\epsilon\right],j,\epsilon\in\mathbb{Z}_{3}.$

(2) $V_{L^{\left(c,j\right)}},j\in\mathbb{Z}_{3}.$

(3) $V_{L}^{T,j}\left(\tau^{k}\right)\left[\epsilon\right],$ $j,\epsilon\in\mathbb{Z}_{3}$
and $k=1,2$.

Weights of these modules are given by
\[
\mbox{wt}V_{L^{\left(0,j\right)}}\left[\epsilon\right]\in\frac{2j^{2}}{3}+\mathbb{Z},
\]

\[
\mbox{wt}V_{L}^{T,j}\left(\tau^{k}\right)\left[\epsilon\right]\in\frac{10-3\left(j^{2}+\epsilon\right)}{9}+\mathbb{Z},
\]
 for $k=1,2$, $j,\epsilon\in\mathbb{Z}_{3}$.

\end{proposition}

Use above notation, we have
\begin{align*}
\mathcal{U} & \cong V_{\mathbb{Z}\beta_{1}}\otimes V_{L}^{\tau}+V_{\frac{1}{3}\beta_{1}+\mathbb{Z}\beta_{1}}\otimes V_{L^{\left(0,1\right)}}\left[0\right]+V_{\frac{2}{3}\beta_{1}+\mathbb{Z}\beta_{1}}\otimes V_{L^{\left(0,2\right)}}\left[0\right].
\end{align*}

Quantum dimensions of irreducible $V_{L}^{\tau}$-modules are obtained
in \cite{C}:

\begin{proposition} \label{quantum dimension of V_Ltau}

(1) qdim$_{V_{L}^{\tau}}V_{L^{\left(0,j\right)}}\left[\epsilon\right]=1$,
for $j,\epsilon\in\mathbb{Z}_{3}.$

(2) qdim$_{V_{L}^{\tau}}V_{L^{\left(c,j\right)}}=3,$ for $j\in\mathbb{Z}_{3}.$

(3) qdim$_{V_{L}^{\tau}}V_{L}^{T,j}\left(\tau^{k}\right)\left[\epsilon\right]=2,$
for $j,\epsilon\in\mathbb{Z}_{3}$ and $k=1,2$.

\end{proposition}

The fusion products among irreducible $V_{L}^{\tau}$-modules are
computed in \cite{TY2,C,CL}:

\begin{proposition} \label{Fusion Product of V_L tau}Let $\epsilon,\epsilon_{1},i,j\in\mathbb{Z}_{3}$
and $k=1,2$.

(i) $V_{L^{\left(0,i\right)}}\left[\epsilon\right]\boxtimes V_{L^{\left(0,j\right)}}\left[\epsilon_{1}\right]=V_{L^{\left(0,i+j\right)}}\left[\epsilon+\epsilon_{1}\right];$

(ii) $V_{L^{\left(0,i\right)}}\left[\epsilon\right]\boxtimes V_{L^{\left(c,j\right)}}=V_{L^{\left(c,i+j\right)}};$

(iii) $V_{L^{\left(c,i\right)}}\boxtimes V_{L^{\left(c,j\right)}}=\sum_{\rho=0}^{2}V_{L^{\left(0,i+j\right)}}\left[\rho\right]+2V_{L^{\left(c,j+j\right)}};$

(iv) $V_{L^{\left(0,i\right)}}\left[\epsilon\right]\boxtimes V_{L}^{T,j}\left(\tau^{k}\right)\left[\epsilon_{1}\right]=V_{L}^{T,j-ki}\left(\tau^{k}\right)\left[k\epsilon+\epsilon_{1}\right];$

(v) $V_{L^{\left(c,i\right)}}\boxtimes V_{L}^{T,j}\left(\tau^{k}\right)\left[\epsilon\right]=\sum_{\rho=0}^{2}V_{L}^{T,j-ki}\left(\tau^{k}\right)\left[\rho\right];$

(vi) $V_{L}^{T,i}\left(\tau^{k}\right)\left[\epsilon\right]\boxtimes V_{L}^{T,j}\left(\tau^{k}\right)\left[\epsilon_{1}\right]=V_{L}^{T,-\left(i+j\right)}\left(\tau^{2k}\right)\left[-\left(\epsilon+\epsilon_{1}\right)\right]+V_{L}^{T,-\left(i+j\right)}\left(\tau^{2k}\right)\left[2-\left(\epsilon+\epsilon_{1}\right)\right];$

(vii) $V_{L}^{T,i}\left(\tau\right)\left[\epsilon\right]\boxtimes V_{L}^{T,j}\left(\tau^{2}\right)\left[\epsilon_{1}\right]=V_{L^{\left(0,i+2j\right)}}\left[\epsilon+2\epsilon_{1}\right]+V_{L^{c,i+2j}}$

\end{proposition}

\begin{remark} \label{Dual of V_L tau modules} Since $V_{L}^{\tau}$
is self-dual, by Proposition \ref{fusion rule symmmetry property}
, we get
\[
1=N_{V_{L}^{\tau}}\left(_{V_{L^{\left(0,i\right)}}\left[\epsilon\right]\ V_{L}^{\tau}}^{V_{L^{\left(0,i\right)}}\left[\epsilon\right]}\right)=N_{V_{L}^{\tau}}\left(_{V_{L^{\left(0,i\right)}}\left[\epsilon\right]\ \left(V_{L^{\left(0,i\right)}}\left[\epsilon\right]\right)^{'}}^{V_{L}^{\tau}}\right).
\]

By fusion rules of irreducible $V_{L}^{\tau}$-modules in Proposition
\ref{Fusion Product of V_L tau}, we see that $\left(V_{L^{\left(0,i\right)}}\left[\epsilon\right]\right)^{'}=V_{L^{\left(0,2i\right)}}\left[2\epsilon\right]$.
Similarly, we can prove $\left(V_{L^{\left(c,i\right)}}\right)^{'}=V_{L^{\left(c,2i\right)}}$
and $\left(V_{L}^{T,i}\left(\tau^{k}\right)\left[\epsilon\right]\right)^{'}=V_{L}^{T,i}\left(\tau^{2k}\right)\left[\epsilon\right]$.

\end{remark}

By Proposition \ref{Fusion Product of V_L tau}, both $V_{L^{\left(0,1\right)}}\left[0\right]$ and $V_{L^{\left(0,2\right)}}\left[0\right]$ are simple currents as $V_L^{\tau}$-modules.  Thus,
$\mathcal{U}$ is a simple current extension of rational vertex operator algebra $V_{\mathbb{Z}\beta_{1}}\otimes V_{L}^{\tau}.$ It follows from \cite{ABD} that $\mathcal{U}$ is $C_2$-cofinite.
The following lemma is clear from \cite{Y} and \cite{HKL}.

\begin{lemma}\label{rationality} $\mathcal{U}$
is a rational, $C_2$-cofinite  vertex operator algebra. \end{lemma}

The $C_2$-cofiniteness and rationality of $\mathcal{U}$ also follow from results on abelian orbifolds \cite{CM, M2}.

Now let $V$ be an arbitrary rational vertex operator algebra We consider the rational vertex operator algebra $V\otimes V\otimes V$
with the natural action of the $3$-cycle $\sigma=\left(1\ 2\ 3\right)$.
The fixed point vertex operator subalgebra is denoted by $\left(V\otimes V\otimes V\right)^{\mathbb{Z}_{3}}.$
Let $M_{i}$, $i=1,\cdots,n$ be all inequivalent irreducible $V$-modules.
For $i,j,k\in\left\{ 1,\cdots,n\right\} $, $M_{i}\otimes M_{j}\otimes M_{k}$
is an irreducible $V\otimes V\otimes V$-module. If $i,j,k$
are not all the same, then $M_{i}\otimes M_{j}\otimes M_{k},$ $M_{j}\otimes M_{k}\otimes M_{i}$,
and $M_{k}\otimes M_{i}\otimes M_{j}$ are isomorphic irreducible
$\left(V\otimes V\otimes V\right)^{\mathbb{Z}_{3}}$-modules
\cite{DM,DY}. The number of such isomorphism classes is $\frac{n^{3}-n}{3}$.

If $i=j=k,$ then $M_{i}\otimes M_{j}\otimes M_{k}+M_{j}\otimes M_{k}\otimes M_{i}+M_{k}\otimes M_{i}\otimes M_{j}$
split into three representations of $\left(V\otimes V\otimes V\right)^{\mathbb{Z}_{3}}$
by \cite{DY}. The number of such isomorphism classes is $3n$.

For $t=1,2$, since there is one-to-one correspondence between the
category of $\sigma^{t}$-twisted $V\otimes V\otimes V$-modules
and the category of $V$-modules \cite{BDM}, the number of isomorphism
classes of irreducible $\sigma^{t}$-twisted $V\otimes V\otimes V$-modules
is also $n$. From \cite{DY}, each $\sigma^{t}$-twisted module can be
decomposed into a direct sum of three irreducible $\left(V\otimes V\otimes V\right)^{\mathbb{Z}_{3}}$-modules.
The number of such isomorphism classes of irreducible $\left(V\otimes V\otimes V\right)^{\mathbb{Z}_{3}}$-modules
$6n.$

In total, we have obtained $\frac{n^{3}+26n}{3}$ inequivalent $\left(V\otimes V\otimes V\right)^{\mathbb{Z}_{3}}$-modules.

In particular, for the case $V=V_{\mathbb{Z}\alpha}$ with $\left\langle \alpha,\alpha\right\rangle =2$,
we have $n=2$. The following result is immediate from Lemma \ref{rationality} and \cite{DRX}:

\begin{proposition}
There are exactly 20 inequivalent irreducible $\mathcal{U}$-modules.
\end{proposition}

In the result of this section, we give explicit realization of irreducible  $\mathcal{U}$-modules in terms of  irreducible $V_{\Z\beta_1}\otimes V_L^{\tau}$-modules.

\begin{proposition} \label{all modules}
An irreducible $\mathcal{U}$-module has the following form
as an $V_{\mathbb{Z}\beta_{1}}\otimes V_{L}^{\tau}$-module: For $i\in\mathbb{Z}_{2}$,
$\epsilon\in\mathbb{Z}_{3}$, $k=1,2$,

\[
M^{i}=V_{\frac{i}{2}\beta_{1}+\mathbb{Z}\beta_{1}}\otimes V_{L^{\left(c,0\right)}}+V_{\frac{3i+2}{6}\beta_{1}+\mathbb{Z}\beta_{1}}\otimes V_{L^{\left(c,1\right)}}+V_{\frac{3i+4}{6}\beta_{1}+\mathbb{Z}\beta_{1}}\otimes V_{L^{\left(c,2\right)}}
\]

\[
\widetilde{M}_{i}\left[\epsilon\right]=V_{\frac{i}{2}\beta_{1}+\mathbb{Z}\beta_{1}}\otimes V_{L^{\left(0,0\right)}}[\epsilon]+V_{\frac{3i+2}{6}\beta_{1}+\mathbb{Z}\beta_{1}}\otimes V_{L^{\left(0,1\right)}}[\epsilon]+V_{\frac{3i+4}{6}\beta_{1}+\mathbb{Z}\beta_{1}}\otimes V_{L^{\left(0,2\right)}}[\epsilon]
\]

\[
\widehat{M}_{\tau^{k},i}\left[\epsilon\right]=V_{\frac{i}{2}\beta_{1}+\mathbb{Z}\beta_{1}}\otimes V_{L}^{T,0}\left(\tau^{k}\right)\left[\epsilon\right]+V_{\frac{3i+2}{6}\beta_{1}+\mathbb{Z}\beta_{1}}\otimes V_{L}^{T,2k}\left(\tau^{k}\right)\left[\epsilon\right]+V_{\frac{3i+4}{6}\beta_{1}+\mathbb{Z}\beta_{1}}\otimes V_{L}^{T,k}\left(\tau^{k}\right)\left[\epsilon\right]
\]

where $\beta_{1}$ can be identified with $\sqrt{3}\alpha$.

\end{proposition}

\begin{proof}

For $i,j,k\in\mathbb{Z}_{2}$, if $i,j$ and $k$ are not all the
same, $\left(i,\  j, \ k\right)=V_{\frac{i}{2}\alpha+\mathbb{Z}\alpha}\otimes V_{\frac{j}{2}\alpha+\mathbb{Z}\alpha}\otimes V_{\frac{k}{2}\alpha+\mathbb{Z}\alpha}+V_{\frac{j}{2}\alpha+\mathbb{Z}\alpha}\otimes V_{\frac{k}{2}\alpha+\mathbb{Z}\alpha}\otimes V_{\frac{i}{2}\alpha+\mathbb{Z}\alpha}+V_{\frac{k}{2}\alpha+\mathbb{Z}\alpha}\otimes V_{\frac{i}{2}\alpha+\mathbb{Z}\alpha}\otimes V_{\frac{j}{2}\alpha+\mathbb{Z}\alpha}$
carries a representation of the $\mathbb{Z}_{3}$ permutation symmetry,
which can be decomposed into three submodules of the orbifold vertex
operator algebra $\left(V_{\mathbb{Z}\alpha}\otimes V_{\mathbb{Z}\alpha}\otimes V_{\mathbb{Z}\alpha}\right)^{\mathbb{Z}_{3}}.$
These three modules are isomorphic as irreducible $\left(V_{\mathbb{Z}\alpha}\otimes V_{\mathbb{Z}\alpha}\otimes V_{\mathbb{Z}\alpha}\right)^{\mathbb{Z}_{3}}$-modules.
We keep only one of these modules. Since
\begin{align*}
 & \left(\frac{i}{2}\alpha^{1}+\mathbb{Z}\alpha^{1}\right)+\left(\frac{j}{2}\alpha^{2}+\mathbb{Z}\alpha^{2}\right)+\left(\frac{k}{2}\alpha^{3}+\mathbb{Z}\alpha^{3}\right)\\
= & \frac{i}{2}\cdot\frac{\beta_{1}+2\beta_{2}+\beta_{3}}{3}+\frac{j}{2}\cdot\frac{\beta_{1}-\beta_{2}+\beta_{3}}{3}+\frac{k}{2}\cdot\frac{\beta_{1}-\beta_{2}-2\beta_{3}}{3}\\
 & +\left(\mathbb{Z}\beta_{1}\oplus L\right)\cup\left(\left(\frac{1}{3}\beta_{1}+\mathbb{Z}\beta_{1}\right)\oplus\left(\frac{2\beta_{2}+\beta_{3}}{3}+L\right)\right)\\
 & +\left(\left(\frac{2}{3}\beta_{1}+\mathbb{Z}\beta_{1}\right)\cup\left(\frac{\beta_{2}+2\beta_{3}}{3}+L\right)\right)\\
= & \left(\left(\frac{i+j+k}{6}\beta_{1}+\mathbb{Z}\beta_{1}\right)\oplus\left(\frac{\left(2i-j-k\right)\beta_{2}+\left(i+j-2k\right)\beta_{3}}{6}+L\right)\right)\\
 & \cup\left(\left(\frac{i+j+k+2}{6}\beta_{1}+\mathbb{Z}\beta_{1}\right)\oplus\left(\frac{\left(2i-j-k+4\right)\beta_{2}+\left(i+j-2k+2\right)\beta_{3}}{6}+L\right)\right)\\
 & \cup\left(\left(\frac{i+j+k+4}{6}\beta_{1}+\mathbb{Z}\beta_{1}\right)\oplus\left(\frac{\left(2i-j-k+2\right)\beta_{2}+\left(i+j-2k+4\right)\beta_{3}}{6}+L\right)\right),
\end{align*}

we obtain
\begin{align*}
\left(i,\ j,\ k\right) & \cong V_{\frac{i+j+k}{6}\beta_{1}+\mathbb{Z}\beta_{1}}\otimes V_{\frac{\left(2i-j-k\right)\beta_{2}+\left(i+j-2k\right)\beta_{3}}{6}+L}\\
 & +V_{\frac{i+j+k+2}{6}\beta_{1}+\mathbb{Z}\beta_{1}}\otimes V_{\frac{\left(2i-j-k+4\right)\beta_{2}+\left(i+j-2k+2\right)\beta_{3}}{6}+L}\\
 & +V_{\frac{i+j+k+4}{6}\beta_{1}+\mathbb{Z}\beta_{1}}\otimes V_{\frac{\left(2i-j-k+2\right)\beta_{2}+\left(i+j-2k+4\right)\beta_{3}}{6}+L}.
\end{align*}
The number of such modules is $\frac{2^{3}-2}{3}=2:$
\begin{align*}
 & \left(0,\ 1,\ 1\right)\cong\left(1,\ 1,\ 0\right)\cong\left(1,\ 0,\ 1\right)\\
\cong & V_{\mathbb{Z}\beta_{1}}\otimes V_{L^{\left(c,0\right)}}+V_{\frac{1}{3}\beta_{1}+\mathbb{Z}\beta_{1}}\otimes V_{L^{\left(c,1\right)}}+V_{\frac{2}{3}\beta_{1}+\mathbb{Z}\beta_{1}}\otimes V_{L^{\left(c,2\right)}}
\end{align*}
and
\begin{align*}
 & \left(1,\ 0,\ 0\right)\cong\left(0,\ 1,\ 0\right)\cong\left(0,\ 0,\ 1\right)\\
\cong & V_{\frac{1}{2}\beta_{1}+\mathbb{Z}\beta_{1}}\otimes V_{L^{\left(c,0\right)}}+V_{\frac{5}{6}\beta_{1}+\mathbb{Z}\beta_{1}}\otimes V_{L^{\left(c,1\right)}}+V_{\frac{1}{6}\beta_{1}+\mathbb{Z}\beta_{1}}\otimes V_{L^{\left(c,2\right).}}
\end{align*}

For simplicity, we denote them by

\[
M^{i}=V_{\frac{i}{2}\beta_{1}+\mathbb{Z}\beta_{1}}\otimes V_{L^{\left(c,0\right)}}+V_{\frac{3i+2}{6}\beta_{1}+\mathbb{Z}\beta_{1}}\otimes V_{L^{\left(c,1\right)}}+V_{\frac{3i+4}{6}\beta_{1}+\mathbb{Z}\beta_{1}}\otimes V_{L^{\left(c,2\right)}},i\in\mathbb{Z}_{2}.
\]

If $i=j=k$, we have

\begin{align*}
\left(i,\ i,\ i\right) & \cong V_{\frac{i}{2}\beta_{1}+\mathbb{Z}\beta_{1}}\otimes V_{L}+V_{\frac{3i+2}{6}\beta_{1}+\mathbb{Z}\beta_{1}}\otimes V_{\frac{2\beta_{2}+\beta_{3}}{3}+L}+V_{\frac{3i+4}{6}\beta_{1}+\mathbb{Z}\beta_{1}}\otimes V_{\frac{\beta_{2}+2\beta_{3}}{3}+L}\\
 & =V_{\frac{i}{2}\beta_{1}+\mathbb{Z}\beta_{1}}\otimes V_{L^{\left(0,0\right)}}+V_{\frac{3i+2}{6}\beta_{1}+\mathbb{Z}\beta_{1}}\otimes V_{L^{\left(0,1\right)}}+V_{\frac{3i+4}{6}\beta_{1}+\mathbb{Z}\beta_{1}}\otimes V_{L^{\left(0,2\right)}}\\
 & =V_{\frac{i}{2}\beta_{1}+\mathbb{Z}\beta_{1}}\otimes\left(V_{L^{\left(0,0\right)}}[0]+V_{L^{\left(0,0\right)}}[1]+V_{L^{\left(0,0\right)}}[2]\right)\\
 & +V_{\frac{3i+2}{6}\beta_{1}+\mathbb{Z}\beta_{1}}\otimes\left(V_{L^{\left(0,1\right)}}[0]+V_{L^{\left(0,1\right)}}[1]+V_{L^{\left(0,1\right)}}[2]\right)\\
 & +V_{\frac{3i+4}{6}\beta_{1}+\mathbb{Z}\beta_{1}}\otimes\left(V_{L^{\left(0,2\right)}}[0]+V_{L^{\left(0,2\right)}}[1]+V_{L^{\left(0,2\right)}}[2]\right)\\
 & =V_{\frac{i}{2}\beta_{1}+\mathbb{Z}\beta_{1}}\otimes V_{L^{\left(0,0\right)}}[0]+V_{\frac{3i+2}{6}\beta_{1}+\mathbb{Z}\beta_{1}}\otimes V_{L^{\left(0,1\right)}}[0]+V_{\frac{3i+4}{6}\beta_{1}+\mathbb{Z}\beta_{1}}\otimes V_{L^{\left(0,2\right)}}[0]\\
 & +V_{\frac{i}{2}\beta_{1}+\mathbb{Z}\beta_{1}}\otimes V_{L^{\left(0,0\right)}}[1]+V_{\frac{3i+2}{6}\beta_{1}+\mathbb{Z}\beta_{1}}\otimes V_{L^{\left(0,1\right)}}[1]+V_{\frac{3i+4}{6}\beta_{1}+\mathbb{Z}\beta_{1}}\otimes V_{L^{\left(0,2\right)}}[1]\\
 & +V_{\frac{i}{2}\beta_{1}+\mathbb{Z}\beta_{1}}\otimes V_{L^{\left(0,0\right)}}[2]+V_{\frac{3i+2}{6}\beta_{1}+\mathbb{Z}\beta_{1}}\otimes V_{L^{\left(0,1\right)}}[2]+V_{\frac{3i+4}{6}\beta_{1}+\mathbb{Z}\beta_{1}}\otimes V_{L^{\left(0,2\right)}}[2].
\end{align*}

Set
\[
\widetilde{M}_{i}\left[\epsilon\right]=V_{\frac{i}{2}\beta_{1}+\mathbb{Z}\beta_{1}}\otimes V_{L^{\left(0,0\right)}}[\epsilon]+V_{\frac{3i+2}{6}\beta_{1}+\mathbb{Z}\beta_{1}}\otimes V_{L^{\left(0,1\right)}}[\epsilon]+V_{\frac{3i+4}{6}\beta_{1}+\mathbb{Z}\beta_{1}}\otimes V_{L^{\left(0,2\right)}}[\epsilon],
\]

then $\widetilde{M}_{i}\left[\epsilon\right]$, $i\in\mathbb{Z}_{2}$
and $\epsilon\in\mathbb{Z}_{3}$ are irreducible modules of $\left(V_{\mathbb{Z}\alpha}\otimes V_{\mathbb{Z}\alpha}\otimes V_{\mathbb{Z}\alpha}\right)^{\mathbb{Z}_{3}}$.
The number of such irreducible modules is $6$.

Note that $V_{\mathbb{Z}\beta_{1}}\otimes V_{L}$ is a subalgebra
of $V_{\Z\alpha}\otimes V_{\Z\alpha}\otimes V_{\Z\alpha}$. So any irreducible $\sigma^{k}$-twisted
$V_{\Z\alpha}\otimes V_{\Z\alpha}\otimes V_{\Z\alpha}$-module contains an irreducible
$1\otimes\tau^{k}$-twisted $V_{\mathbb{Z}\beta_{1}}\otimes V_{L}$-module
$V_{\frac{i}{6}+\mathbb{Z}\beta_{1}}\otimes V_{L}^{T,j}\left(\tau^{k}\right)$
for some $i\in\mathbb{Z}_{6}$, $j\in\mathbb{Z}_{3}$ and $k\in\left\{ 1,2\right\}.$ Moreover, $V_{\frac{i}{6}+\mathbb{Z}\beta_{1}}\otimes V_{L}^{T,j}\left(\tau^{k}\right)$
is a direct sum of three irreducible $V_{\mathbb{Z}\beta_{1}}\otimes V_{L}^{\tau}$-modules
$V_{\frac{i}{6}+\mathbb{Z}\beta_{1}}\otimes V_{L}^{T,j}\left(\tau^{k}\right)$$\left[\epsilon\right],$
$\epsilon\in\mathbb{Z}_{3}$. Since $V_{\mathbb{Z}\beta_{1}}\otimes V_{L}^{\tau}$
is a rational vertex operator subalgebra of $\mathcal{U}$, each irreducible
$\mathcal{U}$-module $M$ is a direct sum of irreducible $V_{\mathbb{Z}\beta_{1}}\otimes V_{L}^{\tau}$-modules.
Let $W$ be an irreducible $V_{\mathbb{Z}\beta_{1}}\otimes V_{L}^{\tau}$-module,
we define $\mathcal{U}\cdot W$ to be the fusion product of $\mathcal{U}$
and $W$ as $V_{\mathbb{Z}\beta_{1}}\otimes V_{L}^{\tau}$-modules.
That is,

\begin{align*}
\mathcal{U}\cdot W & =\left(V_{\mathbb{Z}\beta_{1}}\otimes V_{L}^{\tau}\right)\boxtimes_{V_{\mathbb{Z}\beta_{1}}\otimes V_{L}^{\tau}}W\\
 & \oplus\left(V_{\frac{1}{3}\beta_{1}+\mathbb{Z}\beta_{1}}\otimes V_{L^{\left(0,1\right)}}\left[0\right]\right)\boxtimes_{V_{\mathbb{Z}\beta_{1}}\otimes V_{L}^{\tau}}W\\
 & \oplus\left(V_{\frac{2}{3}\beta_{1}+\mathbb{Z}\beta_{1}}\otimes V_{L^{\left(0,2\right)}}\left[0\right]\right)\boxtimes_{V_{\mathbb{Z}\beta_{1}}\otimes V_{L}^{\tau}}W
\end{align*}
Since both $V_{\frac{1}{3}\beta_{1}+\mathbb{Z}\beta_{1}}\otimes V_{L^{\left(0,1\right)}}\left[0\right]$
and $V_{\frac{2}{3}\beta_{1}+\mathbb{Z}\beta_{1}}\otimes V_{L^{\left(0,2\right)}}\left[0\right]$
are simple current $V_{\mathbb{Z}\beta_{1}}\otimes V_{L}^{\tau}$-modules,
$\mathcal{U}\cdot W$ is a $\mathcal{U}$-module if and only if the
weights of $\mathcal{U}\cdot W$ lies in $\mathbb{Z}+r$ for some
$r\in\mathbb{C}$. From the discussion above, any irreducible $\mathcal{U}$-modules
from the $\sigma^{k}$-twisted $V_{\Z\alpha}\otimes V_{\Z\alpha}\otimes V_{\Z\alpha}$-modules
have the form $\mathcal{U}\cdot W$ where $W=V_{\frac{i}{6}\beta_{1}+\mathbb{Z}\beta_{1}}\otimes V_{L}^{T,j}\left(\tau^{k}\right)\left[\epsilon\right]$
for some $i\in\mathbb{Z}_{6}$, $j,\epsilon\in\mathbb{Z}_{3}$ and
$k\in\left\{ 1,2\right\} $.

Let $W=V_{\frac{i}{6}\beta_{1}+\mathbb{Z}\beta_{1}}\otimes V_{L}^{T,j}\left(\tau^{k}\right)\left[\epsilon\right]$
for some $i\in\mathbb{Z}_{6}$, $j,\epsilon\in\mathbb{Z}_{3}$ and
$k\in\left\{ 1,2\right\} $. Then by fusion rules of irreducible $V_{\mathbb{Z}\beta_{1}}$-
and $V_{L}^{\tau}$-modules \cite{CL,TY2}

\begin{align*}
\mathcal{U}\cdot W & =V_{\frac{i}{6}\beta_{1}+\mathbb{Z}\beta_{1}}\otimes V_{L}^{T,j}\left(\tau^{k}\right)\left[\epsilon\right]\\
 & \oplus V_{\frac{2+i}{6}\beta_{1}+\mathbb{Z}\beta_{1}}\otimes V_{L}^{T,j-k}\left(\tau^{k}\right)\left[\epsilon\right]\\
 & \oplus V_{\frac{4+i}{6}\beta_{1}+\mathbb{Z}\beta_{1}}\otimes V_{L}^{T,j-2k}\left(\tau^{k}\right)\left[\epsilon\right]
\end{align*}

The lowest weight of irreducible $V_{\mathbb{Z}\beta_{1}}$-module
$V_{\lambda+\mathbb{Z}\beta_{1}}$ is $\frac{\left\langle \lambda,\lambda\right\rangle }{2}$,
and the lowest weight of irreducible $V_{L}^{\tau}$-modules of the
form $V_{L}^{T,j}\left(\tau^{k}\right)\left(\epsilon\right)$ is $\frac{10-3\left(j^{2}+\epsilon\right)}{9}$
\cite{TY1}. Thus we see that

(i) If $k=1$, $\mathcal{U}\cdot W$ is a $\mathcal{U}$-module only
if $i\in3\mathbb{Z}$ and $i+2j\in3\mathbb{Z}$. So $\left(i,j\right)=\left(0,0\right)$
or $\left(3,0\right)$. Thus we get two irreducible modules

\[
V_{\mathbb{Z}\beta_{1}}\otimes V_{L}^{T,0}\left(\tau\right)\left[\epsilon\right]+V_{\frac{1}{3}\beta_{1}+\mathbb{Z}\beta_{1}}\otimes V_{L}^{T,2}\left(\tau\right)\left[\epsilon\right]+V_{\frac{2}{3}\beta_{1}+\mathbb{Z}\beta_{1}}\otimes V_{L}^{T,1}\left(\tau\right)\left[\epsilon\right]
\]

and
\[
V_{\frac{1}{2}\beta_{1}+\mathbb{Z}\beta_{1}}\otimes V_{L}^{T,0}\left(\tau\right)\left[\epsilon\right]+V_{\frac{5}{6}\beta_{1}+\mathbb{Z}\beta_{1}}\otimes V_{L}^{T,2}\left(\tau\right)\left[\epsilon\right]+V_{\frac{1}{6}\beta_{1}+\mathbb{Z}\beta_{1}}\otimes V_{L}^{T,1}\left(\tau\right)\left[\epsilon\right].
\]

(ii) If $k=2$, $\mathcal{U}\cdot W$ is a $\mathcal{U}$-module only
if $i\in3\mathbb{Z}-1$ and $i+j\in3\mathbb{Z}$. So $\left(i,j\right)=\left(2,1\right)$
or $\left(5,1\right)$. So we get two irreducible modules

\[
V_{\mathbb{Z}\beta_{1}}\otimes V_{L}^{T,0}\left(\tau^{2}\right)\left[\epsilon\right]+V_{\frac{1}{3}\beta_{1}+\mathbb{Z}\beta_{1}}\otimes V_{L}^{T,1}\left(\tau^{2}\right)\left[\epsilon\right]+V_{\frac{2}{3}\beta_{1}+\mathbb{Z}\beta_{1}}\otimes V_{L}^{T,2}\left(\tau^{2}\right)\left[\epsilon\right]
\]

and
\[
V_{\frac{1}{2}\beta_{1}+\mathbb{Z}\beta_{1}}\otimes V_{L}^{T,1}\left(\tau^{2}\right)\left[\epsilon\right]+V_{\frac{5}{6}\beta_{1}+\mathbb{Z}\beta_{1}}\otimes V_{L}^{T,1}\left(\tau^{2}\right)\left[\epsilon\right]+V_{\frac{1}{6}\beta_{1}+\mathbb{Z}\beta_{1}}\otimes V_{L}^{T,2}\left(\tau^{2}\right)\left[\epsilon\right].
\]

Denote
\[
\widehat{M}_{\tau^{k},i}\left[\epsilon\right]=V_{\frac{i}{2}\beta_{1}+\mathbb{Z}\beta_{1}}\otimes V_{L}^{T,0}\left(\tau^{k}\right)\left[\epsilon\right]+V_{\frac{3i+2}{6}\beta_{1}+\mathbb{Z}\beta_{1}}\otimes V_{L}^{T,2k}\left(\tau^{k}\right)\left[\epsilon\right]+V_{\frac{3i+4}{6}\beta_{1}+\mathbb{Z}\beta_{1}}\otimes V_{L}^{T,k}\left(\tau^{k}\right)\left[\epsilon\right].
\]
 Then the four modules obtained in (i) and (ii) are $\widehat{M}_{\tau,0}\left[\epsilon\right]$,
$\widehat{M}_{\tau,1}\left[\epsilon\right]$, $\widehat{M}_{\tau^{2},0}$$\left[\epsilon\right]$
and $\widehat{M}_{\tau^{2},1}$$\left[\epsilon\right]$ respectively.
The number of irreducible $\mathcal{U}$-modules of the form $\hat{M}_{\tau^{k},i}\left[\epsilon\right]$
with $k=1,2$, $i\in\mathbb{Z}_{2}$, $\epsilon\in\mathbb{Z}_{3}$
is 12.

Now we obtain 20 irreducible $\mathcal{U}$-modules. Thus they are
all the inequivalent irreducible $\mathcal{U}$-modules.

\end{proof}

\begin{remark} It is easy to see that as irreducible $V_{\mathbb{Z}\beta_{1}}$-modules,
$\left(V_{\lambda+\mathbb{Z}\beta_{1}}\right)^{'}=V_{-\lambda+\mathbb{Z}\beta_{1}}$.
By Remark \ref{Dual of V_L tau modules}, we have
\[
\left(M^{i}\right)^{'}=M^{i},\ \left(\widetilde{M}_{i}\left[\epsilon\right]\right)^{'}=\widetilde{M}_{i}\left[2\epsilon\right],\ \mbox{and\ }\left(\widehat{M}_{\tau^{k},i}\left[\epsilon\right]\right)^{'}=\widehat{M}_{\tau^{2k},i}\left[\epsilon\right].
\]

\end{remark}

\section{{\normalsize{}Quantum Dimensions and Fusion Rules}}

In this section, we first give quantum dimensions of all irreducible
$\left(V_{\mathbb{Z}\alpha}\otimes V_{\mathbb{Z}\alpha}\otimes V_{\mathbb{Z}\alpha}\right)^{\mathbb{Z}_{3}}$-modules.
Then we find all fusion products among irreducible $\left(V_{\mathbb{Z}\alpha}\otimes V_{\mathbb{Z}\alpha}\otimes V_{\mathbb{Z}\alpha}\right)^{\mathbb{Z}_{3}}$-modules.

\subsection{{\normalsize{}Quantum Dimensions }}

By Proposition \ref{rationality} and \cite{TY1} we see that the
vertex operator algebra $\left(V_{\mathbb{Z}\alpha}\otimes V_{\mathbb{Z}\alpha}\otimes V_{\mathbb{Z}\alpha}\right)^{\mathbb{Z}_{3}}$
satisfies all the assumptions in the properties of quantum dimensions
in Proposition \ref{possible values of quantum dimensions}. Thus we can use these properties freely.

Quantum dimensions of all irreducible $\left(V_{\mathbb{Z}\alpha}\otimes V_{\mathbb{Z}\alpha}\otimes V_{\mathbb{Z}\alpha}\right)^{\mathbb{Z}_{3}}$-modules
are as follows:

\begin{proposition} \label{quantum dimensions}For $i\in\mathbb{Z}_{2}$,
$\epsilon\in\mathbb{Z}_{3}$, $k=1,2$,  $q\dim_{\mathcal{U}}M^{i}=3,$ $
q\dim_{\mathcal{U}}\widetilde{M}_{i}\left[\epsilon\right]=1,$ $
q\dim_{\mathcal{U}}\widehat{M}_{\tau^{k},i}\left[\epsilon\right]=2.$

\end{proposition}

\begin{proof} For each irreducible $\mathcal{U}$-module $W$, by
definition of quantum dimension, we have
\[
q\dim_{\mathcal{U}}W=\frac{q\dim_{V_{\mathbb{Z}\beta_{1}}\otimes V_{L}^{\tau}}W}{q\dim_{V_{\mathbb{Z}\beta_{1}}\otimes V_{L}^{\tau}}\mathcal{U}}.
\]
Thus we first consider $q\dim_{V_{\mathbb{Z}\beta_{1}}\otimes V_{L}^{\tau}}W.$
By fusion rules of irreducible modules of $V_{\mathbb{Z}\beta_{1}}$-modules
in Proposition \ref{Fusion-V_L}, we see that each $V_{\frac{i}{6}\beta_{1}+\mathbb{Z}\beta_{1}}$,
$i\in\mathbb{Z}_{6}$ is a simple current and hence by Proposition
 \ref{possible values of quantum dimensions}, $q\dim_{V_{\mathbb{Z}\beta_{1}}}V_{\frac{i}{6}\beta_{1}+\mathbb{Z}\beta_{1}}=1$,
$i\in\mathbb{Z}_{6}$. Also by Proposition \ref{quantum dimension of V_Ltau}
and Proposition  \ref{possible values of quantum dimensions}, we get
\[
q\dim_{V_{\mathbb{Z}\beta_{1}}\otimes V_{L}^{\tau}}\mathcal{U}=3\text{, }\ q\dim_{V_{\mathbb{Z}\beta_{1}}\otimes V_{L}^{\tau}}M^{i}=9,\ q\dim_{V_{\mathbb{Z}\beta_{1}}\otimes V_{L}^{\tau}}\widetilde{M}_{i}\left[\epsilon\right]=3,q\dim_{V_{\mathbb{Z}\beta_{1}}\otimes V_{L}^{\tau}}\widehat{M}_{\tau^{k},i}\left[\epsilon\right]=6.
\]
Since $q\dim_{\mathcal{U}}W=\frac{1}{3}q\dim_{V_{\mathbb{Z}\beta_{1}}\otimes V_{L}^{\tau}}W$,
we obtain
\[
q\dim_{\mathcal{U}}M^{i}=3,\ q\dim_{\mathcal{U}}\widetilde{M}_{i}\left[\epsilon\right]=1,\ q\dim_{\mathcal{U}}\widehat{M}_{\tau^{k},i}\left[\epsilon\right]=2.
\]

\end{proof}

\subsection{Fusion Rules among irreducible $\mathcal{U}$-modules}

Now we use the quantum dimensions obtained in the previous subsection
and the fusion rules of irreducible $V_{\mathbb{Z}\beta_{1}}$- and
$V_{L}^{\tau}$-modules in \cite{DL,TY2,C,CL} to determine the fusion
products of the 3-permutation orbifold model.

\begin{theorem}

For $i,j\in\mathbb{Z}_{2}$, $\epsilon,\epsilon_{1}\in\mathbb{Z}_{3},$
$k=1,2$, we have the following fusion product:

\begin{equation}
M^{i}\boxtimes M^{j}\cong\widetilde{M}_{i+j}\left[0\right]+\widetilde{M}_{i+j}\left[1\right]+\widetilde{M}_{i+j}\left[2\right]+2M^{i+j}\label{Fusion-nondiag-nondiag}
\end{equation}

\begin{equation}
M^{i}\boxtimes\widetilde{M}_{j}\left[\epsilon\right]\cong M^{i+j}\label{Fusion-nondiag-diag}
\end{equation}

\begin{equation}
\widetilde{M}_{i}\left[\epsilon\right]\boxtimes\widetilde{M}_{j}\left[\epsilon_{1}\right]\cong\widetilde{M}_{i+j}\left[\epsilon+\epsilon_{1}\right]\label{Fusion-diag-diag}
\end{equation}

\begin{equation}
M^{i}\boxtimes\widehat{M}_{\tau^{k},j}\left[\epsilon\right]\cong\widehat{M}_{\tau^{k},i+j}\left[0\right]+\widehat{M}_{\tau^{k},i+j}\left[1\right]+\widehat{M}_{\tau^{k},i+j}\left[2\right]\label{Fusion-nondiag-twisted}
\end{equation}

\begin{equation}
\widetilde{M}_{i}\left[\epsilon\right]\boxtimes\widehat{M}_{\tau^{k},j}\left[\epsilon\right]\cong\widehat{M}_{\tau^{k},i+j}\left[k\epsilon+\epsilon_{1}\right]\label{Fusion-diag-twisted}
\end{equation}

\begin{equation}
\widehat{M}_{\tau^{k},i}\left[\epsilon\right]\boxtimes\widehat{M}_{\tau^{k},j}\left[\epsilon_{1}\right]\cong\widehat{M}_{\tau^{2k},i+j}\left[-\left(\epsilon+\epsilon_{1}\right)\right]+\widehat{M}_{\tau^{2k},i+j}\left[2-\left(\epsilon+\epsilon_{1}\right)\right]\label{Fusion-twisted-twisted}
\end{equation}

\begin{equation}
\widehat{M}_{\tau^{k},i}\left[\epsilon\right]\boxtimes\widehat{M}_{\tau^{2k},j}\left[\epsilon_{1}\right]\cong\widetilde{M}_{i+j}\left[\epsilon+2\epsilon_{1}\right]+M^{i+j}\label{Fusion-twisted-twisted-case 2}
\end{equation}

where $i+j$ should be understood to be $i+j$ $\mod\ 2$.

\end{theorem}

\begin{proof}

Proof of (\ref{Fusion-nondiag-nondiag}): First by  \ref{quantum dimensions} and Proposition  \ref{possible values of quantum dimensions},
we have
\[
q\dim_{\mathcal{U}}\left(M^{i}\boxtimes M^{j}\right)=q\dim_{\mathcal{U}}M^{i}\cdot q\dim_{\mathcal{U}}M^{j}=9.
\]

By fusion products of irreducible $V_{L}^{\tau}$-modules in Proposition
\ref{Fusion Product of V_L tau}, we know that for any $i,j,l,\epsilon\in\mathbb{Z}_{3}$
and $k\in\left\{ 1,2\right\} $, $N_{V_{L}^{\tau}}\left(_{V_{L^{\left(c,i\right)}\ }V_{L^{\left(c,l\right)}}}^{V_{L}^{T,j}\left(\tau^{k}\right)\left[\epsilon\right]}\right)=0$
. So it is clear that
\[
N_{\mathcal{U}}\left(_{M^{i}\ M^{j}}^{\widehat{M}_{\tau^{k},l}\left[\epsilon\right]}\right)=0\ \mbox{for\ all\ }i,j,l,\epsilon\in\mathbb{Z}_{3},k\in\left\{ 1,2\right\} .
\]

Now we consider $N_{\mathcal{U}}\left(_{M^{i}\ M^{j}}^{M^{k}}\right)$.
Let $V=\mathcal{U}$ and $U=V_{\mathbb{Z}\beta_{1}}\otimes V_{L}^{\tau}$
in Proposition 2.9 \cite{ADL}, we get

\[
N_{\mathcal{U}}\left(_{M^{i}\ M^{j}}^{M^{k}}\right)\le N_{V_{\mathbb{Z}\beta_{1}}\otimes V_{L}^{\tau}}\left(_{V_{\frac{i}{2}\beta_{1}+\mathbb{Z}\beta_{1}}\otimes V_{L^{\left(c,0\right)}}\ V_{\frac{j}{2}\beta_{1}+\mathbb{Z}\beta_{1}}\otimes V_{L^{\left(c,0\right)}}}^{M^{k}}\right).
\]

Since $M^{k}=V_{\frac{k}{2}\beta_{1}+\mathbb{Z}\beta_{1}}\otimes V_{L^{\left(c,0\right)}}+V_{\frac{3k+2}{6}\beta_{1}+\mathbb{Z}\beta_{1}}\otimes V_{L^{\left(c,1\right)}}+V_{\frac{3k+4}{6}\beta_{1}+\mathbb{Z}\beta_{1}}\otimes V_{L^{\left(c,2\right)}}$,
by Proposition 2.10 of \cite{ADL}, we get
\begin{align*}
 & N_{V_{\mathbb{Z}\beta_{1}}\otimes V_{L}^{\tau}}\left(_{V_{\frac{i}{2}\beta_{1}+\mathbb{Z}\beta_{1}}\otimes V_{L^{\left(c,0\right)}}\ V_{\frac{j}{2}\beta_{1}+\mathbb{Z}\beta_{1}}\otimes V_{L^{\left(c,0\right)}}}^{M^{k}}\right)\\
 & =N_{V_{\mathbb{Z}\beta_{1}}\otimes V_{L}^{\tau}}\left(_{V_{\frac{i}{2}\beta_{1}+\mathbb{Z}\beta_{1}}\otimes V_{L^{\left(c,0\right)}}\ V_{\frac{j}{2}\beta_{1}+\mathbb{Z}\beta_{1}}\otimes V_{L^{\left(c,0\right)}}}^{V_{\frac{k}{2}\beta_{1}+\mathbb{Z}\beta_{1}}\otimes V_{L^{\left(c,0\right)}}}\right)\\
 & +N_{V_{\mathbb{Z}\beta_{1}}\otimes V_{L}^{\tau}}\left(_{V_{\frac{i}{2}\beta_{1}+\mathbb{Z}\beta_{1}}\otimes V_{L^{\left(c,0\right)}}\ V_{\frac{j}{2}\beta_{1}+\mathbb{Z}\beta_{1}}\otimes V_{L^{\left(c,0\right)}}}^{V_{\frac{3k+2}{6}\beta_{1}+\mathbb{Z}\beta_{1}}\otimes V_{L^{\left(c,1\right)}}}\right)\\
 & +N_{V_{\mathbb{Z}\beta_{1}}\otimes V_{L}^{\tau}}\left(_{V_{\frac{i}{2}\beta_{1}+\mathbb{Z}\beta_{1}}\otimes V_{L^{\left(c,0\right)}}\ V_{\frac{j}{2}\beta_{1}+\mathbb{Z}\beta_{1}}\otimes V_{L^{\left(c,0\right)}}}^{V_{\frac{3k+4}{6}\beta_{1}+\mathbb{Z}\beta_{1}}\otimes V_{L^{\left(c,2\right)}}}\right)\\
 & =N_{V_{\mathbb{Z}\beta_{1}}}\left(_{V_{\frac{i}{2}\beta_{1}+\mathbb{Z}\beta_{1}}\ V_{\frac{j}{2}\beta_{1}+\mathbb{Z}\beta_{1}}}^{V_{\frac{k}{2}\beta_{1}+\mathbb{Z}\beta_{1}}}\right)\cdot N_{V_{L}^{\tau}}\left(_{V_{L^{\left(c,0\right)}}\ V_{L^{\left(c,0\right)}}}^{V_{L^{\left(c,0\right)}}}\right)\\
 & +N_{V_{\mathbb{Z}\beta_{1}}}\left(_{V_{\frac{i}{2}\beta_{1}+\mathbb{Z}\beta_{1}}\ V_{\frac{j}{2}\beta_{1}+\mathbb{Z}\beta_{1}}}^{V_{\frac{3k+2}{6}\beta_{1}+\mathbb{Z}\beta_{1}}}\right)\cdot N_{V_{L}^{\tau}}\left(_{V_{L^{\left(c,0\right)}}\ V_{L^{\left(c,0\right)}}}^{V_{L^{\left(c,1\right)}}}\right)\\
 & +N_{V_{\mathbb{Z}\beta_{1}}}\left(_{V_{\frac{i}{2}\beta_{1}+\mathbb{Z}\beta_{1}}\ V_{\frac{j}{2}\beta_{1}+\mathbb{Z}\beta_{1}}}^{V_{\frac{3k+4}{6}\beta_{1}+\mathbb{Z}\beta_{1}}}\right)\cdot N_{V_{L}^{\tau}}\left(_{V_{L^{\left(c,0\right)}}\ V_{L^{\left(c,0\right)}}}^{V_{L^{\left(c,2\right)}}}\right).
\end{align*}

By fusion rules of irreducible $V_{L}^{\tau}$-modules in Proposition
\ref{Fusion Product of V_L tau}, we have
\[
N_{V_{L}^{\tau}}\left(_{V_{L^{\left(c,0\right)}}\ V_{L^{\left(c,0\right)}}}^{V_{L^{\left(c,0\right)}}}\right)=2,\ N_{V_{L}^{\tau}}\left(_{V_{L^{\left(c,0\right)}}\ V_{L^{\left(c,0\right)}}}^{V_{L^{\left(c,1\right)}}}\right)=N_{V_{L}^{\tau}}\left(_{V_{L^{\left(c,0\right)}}\ V_{L^{\left(c,0\right)}}}^{V_{L^{\left(c,2\right)}}}\right)=0.
\]

Also note that by fusion rules of irreducible $V_{\mathbb{Z}\beta_{1}}$-modules
in Proposition \ref{Fusion-V_L},
\[
N_{V_{\mathbb{Z}\beta_{1}}}\left(_{V_{\frac{i}{2}\beta_{1}+\mathbb{Z}\beta_{1}}\ V_{\frac{j}{2}\beta_{1}+\mathbb{Z}\beta_{1}}}^{V_{\frac{k}{2}\beta_{1}+\mathbb{Z}\beta_{1}}}\right)=1\ \mbox{only\ if}\ k= i+j\ \mbox{mod}\ 2.
\]
Thus
\[
N_{V_{\mathbb{Z}\beta_{1}}\otimes V_{L}^{\tau}}\left(_{V_{\frac{i}{2}\beta_{1}+\mathbb{Z}\beta_{1}}\otimes V_{L^{\left(c,0\right)}}\ V_{\frac{j}{2}\beta_{1}+\mathbb{Z}\beta_{1}}\otimes V_{L^{\left(c,0\right)}}}^{M^{i+j}}\right)=2\
\]

and

\[
\mbox{\ \ensuremath{N_{V_{\mathbb{Z}\beta_{1}}\otimes V_{L}^{\tau}}\left(_{V_{\frac{i}{2}\beta_{1}+\mathbb{Z}\beta_{1}}\otimes V_{L^{\left(c,0\right)}}\ V_{\frac{j}{2}\beta_{1}+\mathbb{Z}\beta_{1}}\otimes V_{L^{\left(c,0\right)}}}^{M^{k}}\right)=0\ \mbox{if\ }k\not=i+j\ \mbox{mod}\ 2.}}
\]

Thus we obtain
\[
N_{\mathcal{U}}\left(_{M^{i}\ M^{j}}^{M^{i+j}}\right)\le2\ \mbox{and\ }N_{\mathcal{U}}\left(_{M^{i}\ M^{j}}^{M^{k}}\right)=0\ \mbox{if\ }k\not=i+j\ \mbox{mod}\ 2.
\]

Now we consider $N_{\mathcal{U}}\left(_{M^{i}\ M^{j}}^{\widetilde{M}_{k}\left[\epsilon\right]}\right)$.
Similarly, we take $V=\mathcal{U}$ and $U=V_{\mathbb{Z}\beta_{1}}\otimes V_{L}^{\tau}$
in Proposition 2.9 \cite{ADL}, we get

\[
N_{\mathcal{U}}\left(_{M^{i}\ M^{j}}^{\widetilde{M}_{k}\left[\epsilon\right]}\right)\le N_{V_{\mathbb{Z}\beta_{1}}\otimes V_{L}^{\tau}}\left(_{V_{\frac{i}{2}\beta_{1}+\mathbb{Z}\beta_{1}}\otimes V_{L^{\left(c,0\right)}}\ V_{\frac{j}{2}\beta_{1}+\mathbb{Z}\beta_{1}}\otimes V_{L^{\left(c,0\right)}}}^{\widetilde{M}_{k}\left[\epsilon\right]}\right).
\]

Since $\widetilde{M}_{k}\left[\epsilon\right]=V_{\frac{k}{2}\beta_{1}+\mathbb{Z}\beta_{1}}\otimes V_{L^{\left(0,0\right)}}[\epsilon]+V_{\frac{3k+2}{6}\beta_{1}+\mathbb{Z}\beta_{1}}\otimes V_{L^{\left(0,1\right)}}[\epsilon]+V_{\frac{3k+4}{6}\beta_{1}+\mathbb{Z}\beta_{1}}\otimes V_{L^{\left(0,2\right)}}[\epsilon]$,
by Proposition 2.10 of \cite{ADL}, we get
\begin{align*}
 & N_{V_{\mathbb{Z}\beta_{1}}\otimes V_{L}^{\tau}}\left(_{V_{\frac{i}{2}\beta_{1}+\mathbb{Z}\beta_{1}}\otimes V_{L^{\left(c,0\right)}}\ V_{\frac{j}{2}\beta_{1}+\mathbb{Z}\beta_{1}}\otimes V_{L^{\left(c,0\right)}}}^{\widetilde{M}_{k}\left[\epsilon\right]}\right)\\
 & =N_{V_{\mathbb{Z}\beta_{1}}\otimes V_{L}^{\tau}}\left(_{V_{\frac{i}{2}\beta_{1}+\mathbb{Z}\beta_{1}}\otimes V_{L^{\left(c,0\right)}}\ V_{\frac{j}{2}\beta_{1}+\mathbb{Z}\beta_{1}}\otimes V_{L^{\left(c,0\right)}}}^{V_{\frac{k}{2}\beta_{1}+\mathbb{Z}\beta_{1}}\otimes V_{L^{\left(0,0\right)}}[\epsilon]}\right)\\
 & +N_{V_{\mathbb{Z}\beta_{1}}\otimes V_{L}^{\tau}}\left(_{V_{\frac{i}{2}\beta_{1}+\mathbb{Z}\beta_{1}}\otimes V_{L^{\left(c,0\right)}}\ V_{\frac{j}{2}\beta_{1}+\mathbb{Z}\beta_{1}}\otimes V_{L^{\left(c,0\right)}}}^{V_{\frac{3k+2}{6}\beta_{1}+\mathbb{Z}\beta_{1}}\otimes V_{L^{\left(0,1\right)}}[\epsilon]}\right)\\
 & +N_{V_{\mathbb{Z}\beta_{1}}\otimes V_{L}^{\tau}}\left(_{V_{\frac{i}{2}\beta_{1}+\mathbb{Z}\beta_{1}}\otimes V_{L^{\left(c,0\right)}}\ V_{\frac{j}{2}\beta_{1}+\mathbb{Z}\beta_{1}}\otimes V_{L^{\left(c,0\right)}}}^{V_{\frac{3k+4}{6}\beta_{1}+\mathbb{Z}\beta_{1}}\otimes V_{L^{\left(0,2\right)}}[\epsilon]}\right)\\
 & =N_{V_{\mathbb{Z}\beta_{1}}}\left(_{V_{\frac{i}{2}\beta_{1}+\mathbb{Z}\beta_{1}}\ V_{\frac{j}{2}\beta_{1}+\mathbb{Z}\beta_{1}}}^{V_{\frac{k}{2}\beta_{1}+\mathbb{Z}\beta_{1}}}\right)\cdot N_{V_{L}^{\tau}}\left(_{V_{L^{\left(c,0\right)}}\ V_{L^{\left(c,0\right)}}}^{V_{L^{\left(0,0\right)}}[\epsilon]}\right)\\
 & +N_{V_{\mathbb{Z}\beta_{1}}}\left(_{V_{\frac{i}{2}\beta_{1}+\mathbb{Z}\beta_{1}}\ V_{\frac{j}{2}\beta_{1}+\mathbb{Z}\beta_{1}}}^{V_{\frac{3k+2}{6}\beta_{1}+\mathbb{Z}\beta_{1}}}\right)\cdot N_{V_{L}^{\tau}}\left(_{V_{L^{\left(c,0\right)}}\ V_{L^{\left(c,0\right)}}}^{V_{L^{\left(0,1\right)}}[\epsilon]}\right)\\
 & +N_{V_{\mathbb{Z}\beta_{1}}}\left(_{V_{\frac{i}{2}\beta_{1}+\mathbb{Z}\beta_{1}}\ V_{\frac{j}{2}\beta_{1}+\mathbb{Z}\beta_{1}}}^{V_{\frac{3k+4}{6}\beta_{1}+\mathbb{Z}\beta_{1}}}\right)\cdot N_{V_{L}^{\tau}}\left(_{V_{L^{\left(c,0\right)}}\ V_{L^{\left(c,0\right)}}}^{V_{L^{\left(0,2\right)}}[\epsilon]}\right).
\end{align*}

By fusion rules of irreducible $V_{L}^{\tau}$-modules in Proposition
\ref{Fusion Product of V_L tau}, we have

\[
N_{V_{L}^{\tau}}\left(_{V_{L^{\left(c,0\right)}}\ V_{L^{\left(c,0\right)}}}^{V_{L^{\left(0,0\right)}}[\epsilon]}\right)=1\ \mbox{for\ }\epsilon\in\mathbb{Z}_{3},\ \mbox{and\ }N_{V_{L}^{\tau}}\left(_{V_{L^{\left(c,0\right)}}\ V_{L^{\left(c,0\right)}}}^{V_{L^{\left(0,1\right)}}[\epsilon]}\right)=N_{V_{L}^{\tau}}\left(_{V_{L^{\left(c,0\right)}}\ V_{L^{\left(c,0\right)}}}^{V_{L^{\left(0,2\right)}}[\epsilon]}\right)=0.
\]

Also note that by fusion rules of irreducible $V_{\mathbb{Z}\beta_{1}}$-modules
in Proposition \ref{Fusion-V_L},
\[
N_{V_{\mathbb{Z}\beta_{1}}}\left(_{V_{\frac{i}{2}\beta_{1}+\mathbb{Z}\beta_{1}}\ V_{\frac{j}{2}\beta_{1}+\mathbb{Z}\beta_{1}}}^{V_{\frac{k}{2}\beta_{1}+\mathbb{Z}\beta_{1}}}\right)=1\ \mbox{only\ if\ }k=i+j\ \mbox{mod\ }2.
\]
 Thus
\[
N_{V_{\mathbb{Z}\beta_{1}}\otimes V_{L}^{\tau}}\left(_{V_{\frac{i}{2}\beta_{1}+\mathbb{Z}\beta_{1}}\otimes V_{L^{\left(c,0\right)}}\ V_{\frac{j}{2}\beta_{1}+\mathbb{Z}\beta_{1}}\otimes V_{L^{\left(c,0\right)}}}^{\widetilde{M}_{i+j}\left[\epsilon\right]}\right)=1\ \mbox{for\ }\epsilon\in\mathbb{Z}_{3}
\]
 and
\[
N_{V_{\mathbb{Z}\beta_{1}}\otimes V_{L}^{\tau}}\left(_{V_{\frac{i}{2}\beta_{1}+\mathbb{Z}\beta_{1}}\otimes V_{L^{\left(c,0\right)}}\ V_{\frac{j}{2}\beta_{1}+\mathbb{Z}\beta_{1}}\otimes V_{L^{\left(c,0\right)}}}^{\widetilde{M}_{k}\left[\epsilon\right]}\right)=0\ \mbox{if\ }k\not=i+j\ \mbox{mod}\ 2.
\]
 Now we have
\[
N_{\mathcal{U}}\left(_{M^{i}\ M^{j}}^{\widetilde{M}_{i+j}\left[\epsilon\right]}\right)\le1\ \mbox{for\ }\in\mathbb{Z}_{3}\ \mbox{and\ }N_{\mathcal{U}}\left(_{M^{i}\ M^{j}}^{\widetilde{M}_{k}\left[\epsilon\right]}\right)=0\mbox{\ if\ }k\not=i+j\ \mbox{mod}\ 2.
\]

Combine all results above and use quantum dimensions in Proposition
\ref{quantum dimensions}, we see that

\[
N_{\mathcal{U}}\left(_{M^{i}\ M^{j}}^{\widetilde{M}_{i+j}\left[\epsilon\right]}\right)=1\ \mbox{for\ }\epsilon\in\mathbb{Z}_{3}\ \mbox{and}\ N_{\mathcal{U}}\left(_{M^{i}\ M^{j}}^{M^{i+j}}\right)=2.
\]
So we have proved (\ref{Fusion-nondiag-nondiag}).

Proof of (\ref{Fusion-nondiag-diag}): By Propositions \ref{quantum dimensions} and  \ref{possible values of quantum dimensions}, $\widetilde{M}_{j}\left[\epsilon\right]$ is a simple current. From the classification of irreducible $\mathcal{U}$-module, $M^{i}\boxtimes\widetilde{M}_{j}\left[\epsilon\right]=M^{k}$ for
some $k\in\left\{ 0,1\right\}.$ On the other hand, $M^i=(i+1,i+1,i)$ as $V_{\Z\alpha}^{\otimes 3}$-modules and $(i+1,i+1,i)\boxtimes_{V_{\Z\alpha}^{\otimes 3}}(j,j,j)\cong(i+j+1,i+j+1,i+j)=M^{i+j}.$
Let ${\cal Y}\in I_{V_{\Z\alpha}^{\otimes 3}}\left(\begin{array}{c}
M^{i+j}\\
M^i\ (j,j,j)
\end{array}\right)$ be nonzero. The restriction of action of ${\cal Y}$ on $\widetilde{M}_{j}\left[\epsilon\right]$ to $\widetilde{M}_{j}\left[\epsilon\right]$ is an nonzero intertwining operator in $I_{\cal U}\left(\begin{array}{c}
M^{i+j}\\
M^i\ \widetilde{M}_{j}\left[\epsilon\right]
\end{array}\right)$
from \cite{DL}. Thus $k=i+j.$

Proof of (\ref{Fusion-diag-diag}): Since
$q\dim_{\mathcal{U}}\left(\widetilde{M}_i[\epsilon]\boxtimes\widetilde{M}_{j}\left[\epsilon_1\right]\right)=q\dim_{\mathcal{U}}\widetilde{M}_{i}[\epsilon]\cdot q\dim_{\mathcal{U}}\widetilde{M}_{j}\left[\epsilon_1\right]=1,$
$\widetilde{M}_{i}[\epsilon]\boxtimes\widetilde{M}_{j}\left[\epsilon_1\right]=\widetilde{M}_k[\delta]$ for some $k\in \Z_2$ and $\delta\in\Z_3.$ Using $(i,i,i)\boxtimes_{V_{\Z\alpha}^{\otimes 3}}(j,j,j)\cong (i+j,i+j,i+j)$
we see that $k=i+j.$ One could conclude   $\delta=\epsilon+\epsilon_1$ from the explicit expression of the intertwining operator given in \cite{DL}. But we prove this result directly here.
By Proposition 2.9 of \cite{ADL} we see that
\[
N_{\mathcal{U}}\left(_{\widetilde{M}_i[\epsilon]\ \widetilde{M}_{j}\left[\epsilon_1\right]}^{\widetilde{M}_{i+j}[\delta]}\right)\le N_{V_{\mathbb{Z}\beta_{1}}\otimes V_{L}^{\tau}}\left(_{V_{\frac{i}{2}\beta_{1}+\mathbb{Z}\beta_{1}}\otimes V_{L^{\left(0,0\right)}}[\epsilon]\ V_{\frac{j}{2}\beta_{1}+\mathbb{Z}\beta_{1}}\otimes V_{L^{\left(0,0\right)}}[\epsilon_1]}^{\widetilde{M}_{i+j}[\delta]}\right)
\]
which is nonzero if and only if $\delta=\epsilon+\epsilon_1$ by Proposition \ref{Fusion Product of V_L tau}.

Proof of (\ref{Fusion-nondiag-twisted}) and (\ref{Fusion-diag-twisted}): First by fusion rules of irreducible $V_{L}^{\tau}$-modules in
Proposition \ref{Fusion Product of V_L tau} and classification of
irreducible $\mathcal{U}$-modules in Proposition \ref{all modules},
it is easy to see that
\[
N_{_{\mathcal{U}}}\left(_{M^{i}\ \widehat{M}_{\tau^{k},j}\left[\epsilon\right]}^{M^{l}}\right)=N_{\mathcal{U}}\left(_{M^{i}\ \widehat{M}_{\tau^{k},j}\left[\epsilon\right]}^{\widetilde{M}_{n}\left[\epsilon'\right]}\right)=N_{_{\mathcal{U}}}\left(_{M^{i}\ \widehat{M}_{\tau^{k},j}\left[\epsilon\right]}^{\widehat{M}_{\tau^{2k},l}\left[\epsilon'\right]}\right)=0\ \mbox{for\ any\ }l,n\in\mathbb{Z}_{2},\epsilon'\in\mathbb{Z}_{3}.
\]

We first consider $N_{_{V_{\mathbb{Z}\beta_{1}}\otimes V_{L}^{\tau}}}\left(_{M^{i}\ \widehat{M}_{\tau^{k},j}\left[\epsilon\right]}^{\widehat{M}_{\tau^{k},l}\left[\epsilon'\right]}\right)$.
By proposition 2.9 in \cite{ADL} with $V=\mathcal{U}$ and $U=V_{\mathbb{Z}\beta_{1}}\otimes V_{L}^{\tau}$,
we have
\[
N_{_{\mathcal{U}}}\left(_{M^{i}\ \widehat{M}_{\tau^{k},j}\left[\epsilon\right]}^{\widehat{M}_{\tau^{k},l}\left[\epsilon'\right]}\right)\le N_{_{V_{\mathbb{Z}\beta_{1}}\otimes V_{L}^{\tau}}}\left(_{V_{\frac{i}{2}\beta_{1}+\mathbb{Z}\beta_{1}}\otimes V_{L^{\left(c,0\right)}}\ V_{\frac{j}{2}\beta_{1}+\mathbb{Z}\beta_{1}}\otimes V_{L}^{T,0}\left(\tau^{k}\right)\left[\epsilon\right]}^{\widehat{M}_{\tau^{k},l}\left[\epsilon'\right]}\right).
\]

Since $\widehat{M}_{\tau^{k},l}\left[\epsilon'\right]=V_{\frac{l}{2}\beta_{1}+\mathbb{Z}\beta_{1}}\otimes V_{L}^{T,0}\left(\tau^{k}\right)\left[\epsilon'\right]+V_{\frac{3l+2}{6}\beta_{1}+\mathbb{Z}\beta_{1}}\otimes V_{L}^{T,2k}\left(\tau^{k}\right)\left[\epsilon'\right]+V_{\frac{3l+4}{6}\beta_{1}+\mathbb{Z}\beta_{1}}\otimes V_{L}^{T,k}\left(\tau^{k}\right)\left[\epsilon'\right]$

\begin{align*}
 & N_{_{V_{\mathbb{Z}\beta_{1}}\otimes V_{L}^{\tau}}}\left(_{V_{\frac{i}{2}\beta_{1}+\mathbb{Z}\beta_{1}}\otimes V_{L^{\left(c,0\right)}}\ V_{\frac{j}{2}\beta_{1}+\mathbb{Z}\beta_{1}}\otimes V_{L}^{T,0}\left(\tau^{k}\right)\left[\epsilon\right]}^{\widehat{M}_{\tau^{k},l}\left[0\right]}\right)\\
 & =N_{_{V_{\mathbb{Z}\beta_{1}}\otimes V_{L}^{\tau}}}\left(_{V_{\frac{i}{2}\beta_{1}+\mathbb{Z}\beta_{1}}\otimes V_{L^{\left(c,0\right)}}\ V_{\frac{j}{2}\beta_{1}+\mathbb{Z}\beta_{1}}\otimes V_{L}^{T,0}\left(\tau^{k}\right)\left[\epsilon\right]}^{V_{\frac{l}{2}\beta_{1}+\mathbb{Z}\beta_{1}}\otimes V_{L}^{T,0}\left(\tau^{k}\right)\left[\epsilon'\right]}\right)\\
 & +N_{V_{\mathbb{Z}\beta_{1}}\otimes V_{L}^{\tau}}\left(_{V_{\frac{i}{2}\beta_{1}+\mathbb{Z}\beta_{1}}\otimes V_{L^{\left(c,0\right)}}\ V_{\frac{j}{2}\beta_{1}+\mathbb{Z}\beta_{1}}\otimes V_{L}^{T,0}\left(\tau^{k}\right)\left[\epsilon\right]}^{V_{\frac{3l+2}{6}\beta_{1}+\mathbb{Z}\beta_{1}}\otimes V_{L}^{T,2k}\left(\tau^{k}\right)\left[\epsilon'\right]}\right)\\
 & +N_{_{V_{\mathbb{Z}\beta_{1}}\otimes V_{L}^{\tau}}}\left(_{V_{\frac{i}{2}\beta_{1}+\mathbb{Z}\beta_{1}}\otimes V_{L^{\left(c,0\right)}}\ V_{\frac{j}{2}\beta_{1}+\mathbb{Z}\beta_{1}}\otimes V_{L}^{T,0}\left(\tau^{k}\right)\left[\epsilon\right]}^{V_{\frac{3l+4}{6}\beta_{1}+\mathbb{Z}\beta_{1}}\otimes V_{L}^{T,k}\left(\tau^{k}\right)\left[\epsilon'\right]}\right)\\
 & =N_{_{V_{\mathbb{Z}\beta_{1}}}}\left(_{V_{\frac{i}{2}\beta_{1}+\mathbb{Z}\beta_{1}}\ V_{\frac{j}{2}\beta_{1}+\mathbb{Z}\beta_{1}}}^{V_{\frac{l}{2}\beta_{1}+\mathbb{Z}\beta_{1}}}\right)\cdot N_{_{V_{L}^{\tau}}}\left(_{V_{L^{\left(c,0\right)}}\ V_{L}^{T,0}\left(\tau^{k}\right)\left[\epsilon\right]}^{V_{L}^{T,0}\left(\tau^{k}\right)\left[\epsilon'\right]}\right)\\
 & +N_{V_{\mathbb{Z}\beta_{1}}}\left(_{V_{\frac{i}{2}\beta_{1}+\mathbb{Z}\beta_{1}}\ V_{\frac{j}{2}\beta_{1}+\mathbb{Z}\beta_{1}}}^{V_{\frac{3l+2}{6}\beta_{1}+\mathbb{Z}\beta_{1}}}\right)\cdot N_{_{V_{L}^{\tau}}}\left(_{V_{L^{\left(c,0\right)}}\ V_{L}^{T,0}\left(\tau^{k}\right)\left[\epsilon\right]}^{V_{L}^{T,2k}\left(\tau^{k}\right)\left[\epsilon'\right]}\right)\\
 & +N_{V_{\mathbb{Z}\beta_{1}}}\left(_{V_{\frac{i}{2}\beta_{1}+\mathbb{Z}\beta_{1}}\ V_{\frac{j}{2}\beta_{1}+\mathbb{Z}\beta_{1}}}^{V_{\frac{3l+4}{6}\beta_{1}+\mathbb{Z}\beta_{1}}}\right)\cdot N_{_{V_{L}^{\tau}}}\left(_{V_{L^{\left(c,0\right)}}\ V_{L}^{T,0}\left(\tau^{k}\right)\left[\epsilon\right]}^{V_{L}^{T,k}\left(\tau^{k}\right)\left[\epsilon'\right]}\right)
\end{align*}

By fusion rules of irreducible $V_{L}^{\tau}$-modules in Proposition
\ref{Fusion Product of V_L tau},

\[
N_{_{V_{L}^{\tau}}}\left(_{V_{L^{\left(c,0\right)}}\ V_{L}^{T,0}\left(\tau^{k}\right)\left[\epsilon\right]}^{V_{L}^{T,0}\left(\tau^{k}\right)\left[\epsilon'\right]}\right)=1\ \mbox{for\ any\ }\epsilon'\in\mathbb{Z}_{3}
\]
 and
\[
N_{_{V_{L}^{\tau}}}\left(_{V_{L^{\left(c,0\right)}}\ V_{L}^{T,0}\left(\tau^{k}\right)\left[\epsilon\right]}^{V_{L}^{T,2k}\left(\tau^{k}\right)\left[\epsilon'\right]}\right)=N_{_{V_{L}^{\tau}}}\left(_{V_{L^{\left(c,0\right)}}\ V_{L}^{T,0}\left(\tau^{k}\right)\left[\epsilon\right]}^{V_{L}^{T,k}\left(\tau^{k}\right)\left[\epsilon'\right]}\right)=0\ \mbox{\mbox{for\ any\ }\ensuremath{\epsilon}'\ensuremath{\in\mathbb{Z}_{3}}. }
\]

By fusion rules of irreducible $V_{\mathbb{Z}\beta_{1}}$-modules
in Proposition \ref{Fusion-V_L},
\[
N_{_{V_{\mathbb{Z}\beta_{1}}}}\left(_{V_{\frac{i}{2}\beta_{1}+\mathbb{Z}\beta_{1}}\ V_{\frac{j}{2}\beta_{1}+\mathbb{Z}\beta_{1}}}^{V_{\frac{i+j}{2}\beta_{1}+\mathbb{Z}\beta_{1}}}\right)=1\ \mbox{and\ }N_{_{V_{\mathbb{Z}\beta_{1}}}}\left(_{V_{\frac{i}{2}\beta_{1}+\mathbb{Z}\beta_{1}}\ V_{\frac{j}{2}\beta_{1}+\mathbb{Z}\beta_{1}}}^{V_{\frac{l}{2}\beta_{1}+\mathbb{Z}\beta_{1}}}\right)=0\ \mbox{for\ any\ }l\not=i+j\ \mbox{mod}\ 2.
\]

Therefore we obtain
\[
N_{_{V_{\mathbb{Z}\beta_{1}}\otimes V_{L}^{\tau}}}\left(_{V_{\frac{i}{2}\beta_{1}+\mathbb{Z}\beta_{1}}\otimes V_{L^{\left(c,0\right)}}\ V_{\frac{j}{2}\beta_{1}+\mathbb{Z}\beta_{1}}\otimes V_{L}^{T,0}\left(\tau^{k}\right)\left[\epsilon\right]}^{\widehat{M}_{\tau^{k},i+j}\left[\epsilon'\right]}\right)=1\ \mbox{for\ any\ }\epsilon'\in\mathbb{Z}_{3}
\]

and
\[
N_{_{V_{\mathbb{Z}\beta_{1}}\otimes V_{L}^{\tau}}}\left(_{V_{\frac{i}{2}\beta_{1}+\mathbb{Z}\beta_{1}}\otimes V_{L^{\left(c,0\right)}}\ V_{\frac{j}{2}\beta_{1}+\mathbb{Z}\beta_{1}}\otimes V_{L}^{T,0}\left(\tau^{k}\right)\left[\epsilon\right]}^{\widehat{M}_{\tau^{k},l}\left[\epsilon'\right]}\right)=0\ \mbox{if}\ l\not=i+j\ \mbox{mod}\ 2.
\]

This implies
\[
N_{_{V_{\mathbb{Z}\beta_{1}}\otimes V_{L}^{\tau}}}\left(_{M^{i}\ \widehat{M}_{\tau^{k},j}\left[\epsilon\right]}^{\widehat{M}_{\tau^{k},i+j}\left[\epsilon'\right]}\right)\le1\ \mbox{for\ any\ \ensuremath{\epsilon}'\ensuremath{\in\mathbb{Z}_{3}}\ }
\]
and
\[
N_{_{V_{\mathbb{Z}\beta_{1}}\otimes V_{L}^{\tau}}}\left(_{M^{i}\ \widehat{M}_{\tau^{k},j}\left[\epsilon\right]}^{\widehat{M}_{\tau^{k},l}\left[\epsilon'\right]}\right)=0\ \mbox{for\ }l\not=i+j\ \mbox{mod}\ 2.
\]

Combine results above and compare quantum dimensions, we see that
\[
N_{_{V_{\mathbb{Z}\beta_{1}}\otimes V_{L}^{\tau}}}\left(_{M^{i}\ \widehat{M}_{\tau^{k},j}\left[\epsilon\right]}^{\widehat{M}_{\tau^{k},i+j}\left[\epsilon'\right]}\right)=1\ \mbox{for\ any}\ \epsilon'\in\mathbb{Z}_{3}.
\]
So $M^{i}\boxtimes\widehat{M}_{\tau^{k},j}\left[\epsilon\right]\cong \widehat{M}_{\tau^{k},i+j}\left[0\right]+\widehat{M}_{\tau^{k},i+j}\left[1\right]+\widehat{M}_{\tau^{k},i+j}\left[2\right]$.
(\ref{Fusion-diag-twisted}) can be proved similarly.

Proof of (\ref{Fusion-twisted-twisted}) and (\ref{Fusion-twisted-twisted-case 2}):
First by fusion products of irreducible $V_{L}^{\tau}$-modules in
Proposition \ref{Fusion Product of V_L tau}, and classification of
irreducible $\mathcal{U}$-modules in Proposition \ref{all modules},
it is easy to see that for any irreducible $\mathcal{U}$-modules
$W$,
\[
N_{\mathcal{U}}\left(_{\widehat{M}_{\tau^{k},i}\left[\epsilon\right]\ \widehat{M}_{\tau^{k},j}\left[\epsilon_{1}\right]}^{W}\right)\not=0\ \mbox{only\ if\ }W\cong\widehat{M}_{\tau^{2k},l}\left[\epsilon'\right]\ \mbox{for\ some}\ l\in\mathbb{Z}_{2}\ \mbox{and\ }\epsilon'\in\mathbb{Z}_{3}.
\]
 By similar arguments, we can prove the required fusion product. Similarly,
we also can prove (\ref{Fusion-twisted-twisted-case 2}).

\end{proof}

\section{{\normalsize{}$S$-matrix of $\left(V_{L}\otimes V_{L}\otimes V_{L}\right)^{\mathbb{Z}_{3}}$}}

In this section, we use the fusion products of irreducible $\left(V_{L}\otimes V_{L}\otimes V_{L}\right)^{\mathbb{Z}_{3}}$-module
in previous section to obtain the $S$-matrix for $\left(V_{L}\otimes V_{L}\otimes V_{L}\right)^{\mathbb{Z}_{3}}.$

Let $\left(V,Y,1,\omega\right)$ be a vertex operator and $\tau$
be in the complex upper half-plane $\mathbb{H}$ and $q=e^{2\pi i\tau}$.
For any $V$-module $M$ and homogeneous element $v\in V$ we define
a trace function associated to $v$ as follows:

\[
Z_{M}\left(v,\tau\right)=\mbox{tr}_{M}o\left(v\right)q^{L\left(0\right)-c/24}=q^{\lambda-c/24}\sum_{n\in\mathbb{Z}_{+}}\mbox{tr}_{M_{\lambda+n}}o\left(v\right)q^{n},
\]
 where $o\left(v\right)=v\left(\wt v-1\right)$ is the degree
zero operator of $v$. In \cite{Z}, Zhu has introduced a second vertex
operator algebra $\left(V,Y\left[\cdot,\cdot\right],1,\tilde{\omega}\right)$
associated to $V$. Here $\tilde{\omega}=\omega-c/24$ and
\[
Y\left[v,z\right]=Y\left(v,e^{z}-1\right)e^{z\cdot\wt v}=\sum_{n\in\mathbb{Z}}v\left[n\right]z^{-n-1}
\]
 for homogeneous $v$. We also write
\[
Y\left[\tilde{\omega},z\right]=\sum_{n\in\mathbb{Z}}L\left[n\right]z^{n-2}.
\]
If $v\in V$ is homogeneous in the second vertex operator algebra,
we denote its weight by $\wt[v]$. The following result is from \cite{Z} and \cite{DLN}.

\begin{theorem} Assume $V$ is a rational and $C_{2}$-cofinite vertex
operator algebra and $W^{0},W^{1},\cdots,W^{n}$ are irreducible $V$-modules,
then
\[
Z_{W^{i}}\left(v,-\frac{1}{\tau}\right)=\tau^{\wt v}\sum_{j=0}^{n}S_{i,j}Z_{W^{j}}\left(v,\tau\right).
\]
Moreover, each $Z_{W^{i}}(v,\tau)$ is a modular form of weight $\wt v.$ The matrix $S=\left(S_{i,j}\right)$ is called an $S$-matrix which
is independent of the choice of $v$.
\end{theorem}

The following properties of $S$-matrix \cite{H1, H2} can reduce calculation of
entries of $S$-matrix.

\begin{theorem} \label{property of S-matrix}Let $V$ be a rational
and $C_{2}$-cofinite simple vertex operator algebra of CFT type and
assume $V\cong V'$. Then the $V$-module category is a modular tensor category.  Let $S=\left(S_{i,j}\right)_{i,j=0}^{n}$ be
the $S$-matrix as defined above. Then

(i) $\left(S^{-1}\right)_{i,j}=S_{i,j'}=S_{i',j}$, and $S_{i',j'}=S_{i,j}$.

(ii) $S$ is symmetric and $S^{2}=\left(\delta_{i,j'}\right)$.

\end{theorem}

For convenience, we denote the irreducible $\left(V_{L}\otimes V_{L}\otimes V_{L}\right)^{\mathbb{Z}_{3}}$-modules
by $W^{i},i=0,1,\cdots,19$ as following:

\noindent \begin{center}
\begin{tabular}{|c|c|c|c|c|c|}
\hline
$W^{0}$  & $W^{1}$  & $W^{2}$  & $W^{3}$  & $W^{4}$  & $W^{5}$\tabularnewline
\hline
\hline
$\widetilde{M}_{0}\mbox{\ensuremath{\left[0\right]}}$  & $\widetilde{M}_{0}\mbox{\ensuremath{\left[1\right]}}$  & $\widetilde{M}_{0}\mbox{\ensuremath{\left[2\right]}}$  & $\widetilde{M}_{1}\mbox{\ensuremath{\left[0\right]}}$  & $\widetilde{M}_{1}\mbox{\ensuremath{\left[1\right]}}$  & $\widetilde{M}_{1}\mbox{\ensuremath{\left[2\right]}}$\tabularnewline
\hline
\end{tabular}
\par\end{center}

\noindent \begin{center}
\begin{tabular}{|c|c|}
\hline
$W^{6}$  & $W^{7}$\tabularnewline
\hline
\hline
$M^{0}$  & $M^{1}$\tabularnewline
\hline
\end{tabular}
\par\end{center}

\noindent \begin{center}
\begin{tabular}{|c|c|c|c|c|c|}
\hline
$W^{8}$  & $W^{9}$  & $W^{10}$  & $W^{11}$  & $W^{12}$  & $W^{13}$\tabularnewline
\hline
\hline
$\widehat{M}_{\tau,0}\left[0\right]$  & $\widehat{M}_{\tau,0}\left[1\right]$  & $\widehat{M}_{\tau,0}\left[2\right]$  & $\widehat{M}_{\tau,1}\left[0\right]$  & $\widehat{M}_{\tau,1}\left[1\right]$  & $\widehat{M}_{\tau,1}\left[2\right]$\tabularnewline
\hline
\end{tabular}
\par\end{center}

\noindent \begin{center}
\begin{tabular}{|c|c|c|c|c|c|}
\hline
$W^{14}$  & $W^{15}$  & $W^{16}$  & $W^{17}$  & $W^{18}$  & $W^{19}$\tabularnewline
\hline
\hline
$\widehat{M}_{\tau^{2},0}\left[0\right]$  & $\widehat{M}_{\tau^{2},0}\left[1\right]$  & $\widehat{M}_{\tau^{2},0}\left[2\right]$  & $\widehat{M}_{\tau^{2},1}\left[0\right]$  & $\widehat{M}_{\tau^{2},1}\left[1\right]$  & $\widehat{M}_{\tau^{2},1}\left[2\right]$\tabularnewline
\hline
\end{tabular}
\par\end{center}

The following table gives the conformal weight and quantum dimension
of each $W^{i},i=0,1,\cdots,19$:

\begin{tabular}{|c|c|c|c|c|c|c|c|c|c|c|c|}
\hline
$W^{i}$ & $W^{0}$  & $W^{1}$  & $W^{2}$  & $W^{3}$  & $W^{4}$  & $W^{5}$ & $W^{6}$  & $W^{7}$ & $W^{8}$  & $W^{9}$  & $W^{10}$ \tabularnewline
\hline
\hline
conformal weight & $0$  & $1$  & $1$  & $\frac{3}{4}$  & $\frac{3}{4}$  & $\frac{3}{4}$ & $\frac{1}{2}$  & $\frac{1}{4}$ & $\frac{1}{9}$  & $\frac{7}{9}$  & $\frac{4}{9}$ \tabularnewline
\hline
quantum dimension & 1 & 1 & 1 & 1 & 1 & 1 & 3 & 3 & 2 & 2 & 2\tabularnewline
\hline
\end{tabular}

\begin{tabular}{|c|c|c|c|c|c|c|c|c|c|}
\hline
$W^{i}$ & $W^{11}$  & $W^{12}$  & $W^{13}$ & $W^{14}$  & $W^{15}$  & $W^{16}$  & $W^{17}$  & $W^{18}$  & $W^{19}$\tabularnewline
\hline
\hline
conformal weight & $\frac{31}{36}$  & $\frac{19}{36}$  & $\frac{7}{36}$ & $\frac{1}{9}$  & $\frac{7}{9}$  & $\frac{4}{9}$  & $\frac{31}{36}$  & $\frac{19}{36}$  & $\frac{7}{36}$\tabularnewline
\hline
quantum dimension & 2 & 2 & 2 & 2 & 2 & 2 & 2 & 2 & 2\tabularnewline
\hline
\end{tabular}

Since the $\left(V_{L}\otimes V_{L}\otimes V_{L}\right)^{\mathbb{Z}_{3}}$-module category is a modular tensor category from Theorem \ref{property of S-matrix},  by (3.1.16) in \cite{BK} we have $S=\frac{\tilde{s}}{D}$ where $D$ is the square root
of $\mbox{glob}\left(V_{L}\otimes V_{L}\otimes V_{L}\right)^{\mathbb{Z}_{3}}=72$
and $\tilde{s}=\left(\tilde{s}_{i,j}\right)$.  So $S_{i,j}=\frac{1}{6\sqrt{2}}\tilde{s}_{i,j}$. 
From \cite{BK}
\begin{equation}
\tilde{s}_{i,j}=\sum_{k=0}^{19}N_{i',j}^{k}\cdot\frac{\theta_{k}}{\theta_{i}\theta_{j}}\cdot q\dim_{\mathcal{U}}W^{k}\label{The matrix Y}
\end{equation}
where $\theta_{i}=e^{2\pi i\Delta_{i}}$ and $\Delta_{i}$ is the
conformal weight of $W^{i}$.

\begin{lemma} We also have
$$\tilde{s}_{i,j}=\sum_{k=0}^{19}N_{i,j}^{k}\cdot\frac{\theta_{i}\theta_{j}}{\theta_{k}}\cdot q\dim_{\mathcal{U}}W^{k}.$$
\end{lemma}
\begin{proof} From Theorem \ref{property of S-matrix} and the fact that $S$ is a unitary matrix \cite{DLN}, we know that $S_{i,j}=\overline{S_{i',j}}.$ Since $D$ is a real number we see that $\tilde{s}_{i,j}=\overline{\tilde{s}_{i',j}}.$ Notethat  $\theta_i, \theta_j,\theta_k$
are roots of unity and $\theta_i=\theta_{i'}.$  The identity $\tilde{s}_{i,j}=\sum_{k=0}^{19}N_{i,j}^{k}\cdot\frac{\theta_{i}\theta_{j}}{\theta_{k}}\cdot q\dim_{\mathcal{U}}W^{k}$ follows immediately.
\end{proof}

By properties of $S$-matrix in Proposition \ref{property of S-matrix},
we see that dual property of modules can reduce the calculation of
entries of the $S$-matrix.

\begin{lemma}The dual module of $W^{i}$, $i=0,1,\cdots,19$ is given
by the following.

\[
\left(W^{0}\right)^{'}=W^{0};\ \left(W^{1}\right)^{'}=W^{2};\ \left(W^{2}\right)^{'}=W^{1};\ \left(W^{3}\right)^{'}=W^{3};\ \left(W^{4}\right)^{'}=W^{5};\ \left(W^{5}\right)^{'}=W^{4};
\]

\[
\left(W^{6}\right)^{'}=W^{6};\ \left(W^{7}\right)^{'}=W^{7};\ \left(W^{8}\right)^{'}=W^{14};\ \left(W^{9}\right)^{'}=W^{15};\ \left(W^{10}\right)^{'}=W^{16};\ \left(W^{11}\right)^{'}=W^{17};
\]

\[
\left(W^{12}\right)^{'}=W^{18};\ \left(W^{13}\right)^{'}=W^{19};\ \left(W^{14}\right)^{'}=W^{8};\ \left(W^{15}\right)^{'}=W^{9};\ \left(W^{16}\right)^{'}=W^{10};
\]

\[
\left(W^{17}\right)^{'}=W^{11};\ \left(W^{18}\right)^{'}=W^{12};\ \left(W^{19}\right)^{'}=W^{13}.
\]

\end{lemma}

\begin{proposition}The entries of the $S$-matrix of the orbifold
vertex operator algebra $\left(V_{L}\otimes V_{L}\otimes V_{L}\right)^{\mathbb{Z}_{3}}$
is given by the following table.

\end{proposition}

\begin{proof} By the fusion products in Proposition \ref{Fusion Product of V_L tau},
we have $N_{i,j}^{k}$ for any $i,j,k\in\left\{ 0,1,\cdots,19\right\} .$
We also have confomal weights $\Delta_{i}$ of each $W^{i}$ and quantum
dimensions $q\dim_{\mathcal{U}}W^{i}$ given in Proposition \ref{quantum dimensions}.
By \ref{The matrix Y}, entries of $Y$ is straightforward. \end{proof}

The following is the matrix $\left(\tilde{s}_{i,j}\right)_{i,j=0}^{19}$:

\begin{center}
\begin{tabular}{|c|c|c|c|c|c|c|c|c|}
\hline
$6\sqrt{2}S_{i,j}$  & 0  & 1  & 2  & 3  & 4  & 5  & 6  & 7 \tabularnewline
\hline
\hline
0  & $1$  & $1$  & $1$  & $1$  & $1$  & $1$  & $3$  & $3$ \tabularnewline
\hline
1  & $1$  & $1$  & $1$  & $1$  & $1$  & $1$  & $3$  & $3$ \tabularnewline
\hline
2  & $1$  & $1$  & $1$  & $1$  & $1$  & $1$  & $3$  & $3$ \tabularnewline
\hline
3  & $1$  & $1$  & $1$  & $-1$  & $-1$  & $-1$  & $3$  & $-3$ \tabularnewline
\hline
4  & $1$  & $1$  & $1$  & $-1$  & $-1$  & $-1$  & $3$  & $-3$ \tabularnewline
\hline
5  & $1$  & $1$  & $1$  & $-1$  & $-1$  & $-1$  & $3$  & $-3$ \tabularnewline
\hline
6  & $3$  & $3$  & $3$  & $3$  & $3$  & $3$  & -3  & -3 \tabularnewline
\hline
7  & $3$  & $3$  & $3$  & $-3$  & $-3$  & $-3$  & $-3$  & $3$ \tabularnewline
\hline
8  & $2$  & $2e^{\frac{2\pi i}{3}}$  & $2e^{-\frac{2\pi i}{3}}$  & $2$  & $2e^{\frac{2\pi i}{3}}$  & $2e^{-\frac{2\pi i}{3}}$  & $0$  & $0$ \tabularnewline
\hline
9  & $2$  & $2e^{\frac{2\pi i}{3}}$  & $2e^{-\frac{2\pi i}{3}}$  & $2$  & $2e^{\frac{2\pi i}{3}}$  & $2e^{-\frac{2\pi i}{3}}$  & $0$  & $0$ \tabularnewline
\hline
10  & $2$  & $2e^{\frac{2\pi i}{3}}$  & $2e^{-\frac{2\pi i}{3}}$  & $2$  & $2e^{\frac{2\pi i}{3}}$  & $2e^{-\frac{2\pi i}{3}}$  & $0$  & $0$ \tabularnewline
\hline
11  & $2$  & $2e^{\frac{2\pi i}{3}}$  & $2e^{-\frac{2\pi i}{3}}$  & $-2$  & $2e^{-\frac{\pi i}{3}}$  & $2e^{\frac{\pi i}{3}}$  & $0$  & $0$ \tabularnewline
\hline
12  & 2  & $2e^{\frac{2\pi i}{3}}$  & $2e^{-\frac{2\pi i}{3}}$  & $-2$  & $2e^{-\frac{\pi i}{3}}$  & $2e^{\frac{\pi i}{3}}$  & $0$  & $0$ \tabularnewline
\hline
13  & $2$  & $2e^{\frac{2\pi i}{3}}$  & $2e^{-\frac{2\pi i}{3}}$  & $-2$  & $2e^{-\frac{\pi i}{3}}$  & $2e^{\frac{\pi i}{3}}$  & $0$  & $0$ \tabularnewline
\hline
14  & $2$  & $2e^{-\frac{2\pi i}{3}}$  & $2e^{\frac{2\pi i}{3}}$  & $2$  & $2e^{-\frac{2\pi i}{3}}$  & $2e^{\frac{2\pi i}{3}}$  & $0$  & $0$ \tabularnewline
\hline
15  & $2$  & $2e^{-\frac{2\pi i}{3}}$  & $2e^{\frac{2\pi i}{3}}$  & $2$  & $2e^{-\frac{2\pi i}{3}}$  & $2e^{\frac{2\pi i}{3}}$  & $0$  & $0$ \tabularnewline
\hline
16  & $2$  & $2e^{-\frac{2\pi i}{3}}$  & $2e^{\frac{2\pi i}{3}}$  & $2$  & $2e^{-\frac{2\pi i}{3}}$  & $2e^{\frac{2\pi i}{3}}$  & $0$  & $0$ \tabularnewline
\hline
17  & $2$  & $2e^{-\frac{2\pi i}{3}}$  & $2e^{\frac{2\pi i}{3}}$  & $-2$  & $2e^{\frac{\pi i}{3}}$  & $2e^{-\frac{\pi i}{3}}$  & $0$  & $0$ \tabularnewline
\hline
18  & 2  & $2e^{-\frac{2\pi i}{3}}$  & $2e^{\frac{2\pi i}{3}}$  & $-2$  & $2e^{\frac{\pi i}{3}}$  & $2e^{-\frac{\pi i}{3}}$  & $0$  & $0$ \tabularnewline
\hline
19  & $2$  & $2e^{-\frac{2\pi i}{3}}$  & $2e^{\frac{2\pi i}{3}}$  & $-2$  & $2e^{\frac{\pi i}{3}}$  & $2e^{-\frac{\pi i}{3}}$  & $0$  & $0$ \tabularnewline
\hline
\end{tabular}
\par\end{center}

\begin{center}
\begin{tabular}{|c|c|c|c|c|c|}
\hline
$6\sqrt{2}S_{i,j}$  & 8  & 9  & 10 & 11  & 12 \tabularnewline
\hline
\hline
0  & 2  & $2$  & $2$ & $2$  & $2$ \tabularnewline
\hline
1  & $2e^{\frac{2\pi i}{3}}$  & $2e^{\frac{2\pi i}{3}}$  & $2e^{\frac{2\pi i}{3}}$ & $2e^{\frac{2\pi i}{3}}$  & $2e^{\frac{2\pi i}{3}}$ \tabularnewline
\hline
2  & $2e^{\frac{-2\pi i}{3}}$  & $2e^{-\frac{2\pi i}{3}}$  & $2e^{-\frac{2\pi i}{3}}$ & $2e^{-\frac{2\pi i}{3}}$  & $2e^{-\frac{2\pi i}{3}}$ \tabularnewline
\hline
3  & $2$  & $2$  & $2$ & $-2$  & $-2$ \tabularnewline
\hline
4  & $2e^{\frac{2\pi i}{3}}$  & $2e^{\frac{2\pi i}{3}}$  & $2e^{\frac{2\pi i}{3}}$ & $2e^{-\frac{\pi i}{3}}$  & $2e^{-\frac{\pi i}{3}}$ \tabularnewline
\hline
5  & $2e^{-\frac{2\pi i}{3}}$  & $2e^{-\frac{2\pi i}{3}}$  & $2e^{-\frac{2\pi i}{3}}$ & $2e^{\frac{\pi i}{3}}$  & $2e^{\frac{\pi i}{3}}$ \tabularnewline
\hline
6  & $0$  & $0$  & $0$ & $0$  & $0$ \tabularnewline
\hline
7  & $0$  & $0$  & $0$ & $0$  & $0$ \tabularnewline
\hline
8  & $2e^{-\frac{4\pi i}{9}}+2e^{\frac{2\pi i}{9}}$  & $2e^{\frac{2\pi i}{9}}+2e^{\frac{8\pi i}{9}}$  & $2e^{-\frac{4\pi i}{9}}+2e^{\frac{8\pi i}{9}}$ & $2e^{-\frac{4\pi i}{9}}+2e^{\frac{2\pi i}{9}}$  & $2e^{\frac{2\pi i}{9}}+2e^{\frac{8\pi i}{9}}$ \tabularnewline
\hline
9  & $2e^{\frac{2\pi i}{9}}+2e^{\frac{8\pi i}{9}}$  & $2e^{-\frac{4\pi i}{9}}+2e^{\frac{8\pi i}{9}}$  & $2e^{-\frac{4\pi i}{9}}+2e^{\frac{2\pi i}{9}}$ & $2e^{\frac{2\pi i}{9}}+2e^{\frac{8\pi i}{9}}$  & $2e^{-\frac{4\pi i}{9}}+2e^{\frac{8\pi i}{9}}$ \tabularnewline
\hline
10  & $2e^{-\frac{4\pi i}{9}}+2e^{\frac{8\pi i}{9}}$  & $2e^{-\frac{4\pi i}{9}}+2e^{\frac{2\pi i}{9}}$  & $2e^{\frac{2\pi i}{9}}+2e^{\frac{8\pi i}{9}}$ & $2e^{-\frac{4\pi i}{9}}+2e^{\frac{8\pi i}{9}}$  & $2e^{-\frac{4\pi i}{9}}+2e^{\frac{2\pi i}{9}}$ \tabularnewline
\hline
11  & $2e^{-\frac{4\pi i}{9}}+2e^{\frac{2\pi i}{9}}$  & $2e^{\frac{2\pi i}{9}}+2e^{\frac{8\pi i}{9}}$  & $2e^{-\frac{4\pi i}{9}}+2e^{\frac{8\pi i}{9}}$ & $2e^{-\frac{7\pi i}{9}}+2e^{\frac{5\pi i}{9}}$  & $2e^{-\frac{\pi i}{9}}+2e^{-\frac{7\pi i}{9}}$ \tabularnewline
\hline
12  & $2e^{\frac{2\pi i}{9}}+2e^{\frac{8\pi i}{9}}$  & $2e^{-\frac{4\pi i}{9}}+2e^{\frac{8\pi i}{9}}$  & $2e^{-\frac{4\pi i}{9}}+2e^{\frac{2\pi i}{9}}$ & $2e^{-\frac{\pi i}{9}}+2e^{-\frac{7\pi i}{9}}$  & $2e^{-\frac{\pi i}{9}}+2e^{\frac{5\pi i}{9}}$ \tabularnewline
\hline
13  & $2e^{-\frac{4\pi i}{9}}+2e^{\frac{8\pi i}{9}}$  & $2e^{-\frac{4\pi i}{9}}+2e^{\frac{2\pi i}{9}}$  & $2e^{\frac{2\pi i}{9}}+2e^{\frac{8\pi i}{9}}$ & $2e^{-\frac{\pi i}{9}}+2e^{\frac{5\pi i}{9}}$  & $2e^{-\frac{7\pi i}{9}}+2e^{\frac{5\pi i}{9}}$ \tabularnewline
\hline
14  & $3e^{-\frac{5\pi i}{9}}+e^{\frac{4\pi i}{9}}$  & $e^{-\frac{2\pi i}{9}}+3e^{\frac{7\pi i}{9}}$  & $e^{-\frac{8\pi i}{9}}+3e^{\frac{\pi i}{9}}$ & $3e^{-\frac{5\pi i}{9}}+e^{\frac{4\pi i}{9}}$ & $e^{-\frac{2\pi i}{9}}+3e^{\frac{7\pi i}{9}}$ \tabularnewline
\hline
15  & $e^{-\frac{2\pi i}{9}}+3e^{\frac{7\pi i}{9}}$  & $e^{-\frac{8\pi i}{9}}+3e^{\frac{\pi i}{9}}$  & $e^{\frac{4\pi i}{9}}+3e^{-\frac{5\pi i}{9}}$ & $e^{-\frac{2\pi i}{9}}+3e^{\frac{7\pi i}{9}}$  & $e^{-\frac{8\pi i}{9}}+3e^{\frac{\pi i}{9}}$ \tabularnewline
\hline
16  & $e^{-\frac{8\pi i}{9}}+3e^{\frac{\pi i}{9}}$  & $3e^{-\frac{5\pi i}{9}}+e^{\frac{4\pi i}{9}}$ & $e^{-\frac{2\pi i}{9}}+3e^{\frac{7\pi i}{9}}$ & $e^{-\frac{8\pi i}{9}}+3e^{\frac{\pi i}{9}}$  & $3e^{-\frac{5\pi i}{9}}+e^{\frac{4\pi i}{9}}$ \tabularnewline
\hline
17  & $3e^{-\frac{5\pi i}{9}}+e^{\frac{4\pi i}{9}}$ & $e^{-\frac{2\pi i}{9}}+3e^{\frac{7\pi i}{9}}$  & $e^{-\frac{8\pi i}{9}}+3e^{\frac{\pi i}{9}}$ & $3e^{-\frac{5\pi i}{9}}+e^{\frac{4\pi i}{9}}$  & $3e^{-\frac{2\pi i}{9}}+e^{\frac{7\pi i}{9}}$ \tabularnewline
\hline
18  & $e^{-\frac{2\pi i}{9}}+3e^{\frac{7\pi i}{9}}$  & $e^{-\frac{8\pi i}{9}}+3e^{\frac{\pi i}{9}}$  & $3e^{-\frac{5\pi i}{9}}+e^{\frac{4\pi i}{9}}$ & $e^{-\frac{2\pi i}{9}}+3e^{\frac{7\pi i}{9}}$  & $3e^{-\frac{8\pi i}{9}}+e^{\frac{\pi i}{9}}$ \tabularnewline
\hline
19  & $e^{-\frac{8\pi i}{9}}+3e^{\frac{\pi i}{9}}$  & $3e^{-\frac{5\pi i}{9}}+e^{\frac{4\pi i}{9}}$ & $e^{-\frac{2\pi i}{9}}+3e^{\frac{7\pi i}{9}}$ & $e^{-\frac{8\pi i}{9}}+3e^{\frac{\pi i}{9}}$  & $e^{-\frac{5\pi i}{9}}+3e^{\frac{4\pi i}{9}}$ \tabularnewline
\hline
\end{tabular}
\par\end{center}

\begin{center}
\begin{tabular}{|c|c|c|c|c|}
\hline
$6\sqrt{2}S_{i,j}$  & 13  & 14  & 15  & 16 \tabularnewline
\hline
\hline
0  & $2$  & $2$  & $2$  & $2$ \tabularnewline
\hline
1  & $2e^{\frac{2\pi i}{3}}$  & $2e^{-\frac{2\pi i}{3}}$  & $2e^{-\frac{2\pi i}{3}}$  & $2e^{-\frac{2\pi i}{3}}$ \tabularnewline
\hline
2  & $2e^{-\frac{2\pi i}{3}}$  & $2e^{\frac{2\pi i}{3}}$  & $2e^{\frac{2\pi i}{3}}$  & $2e^{\frac{2\pi i}{3}}$ \tabularnewline
\hline
3  & $-2$  & $2$  & $2$  & $2$ \tabularnewline
\hline
4  & $2e^{-\frac{\pi i}{3}}$  & $2e^{-\frac{2\pi i}{3}}$  & $2e^{-\frac{2\pi i}{3}}$  & $2e^{-\frac{2\pi i}{3}}$ \tabularnewline
\hline
5  & $2e^{\frac{\pi i}{3}}$  & $2e^{\frac{2\pi i}{3}}$  & $2e^{\frac{2\pi i}{3}}$  & $2e^{\frac{2\pi i}{3}}$ \tabularnewline
\hline
6  & $0$  & $0$  & $0$  & $0$ \tabularnewline
\hline
7  & $0$  & $0$  & $0$  & $0$ \tabularnewline
\hline
8  & $2e^{-\frac{4\pi i}{9}}+2e^{\frac{8\pi i}{9}}$  & $3e^{-\frac{5\pi i}{9}}+e^{\frac{4\pi i}{9}}$  & $e^{-\frac{2\pi i}{9}}+3e^{\frac{7\pi i}{9}}$  & $e^{-\frac{8\pi i}{9}}+3e^{\frac{\pi i}{9}}$ \tabularnewline
\hline
9  & $2e^{-\frac{4\pi i}{9}}+2e^{\frac{2\pi i}{9}}$  & $e^{-\frac{2\pi i}{9}}+3e^{\frac{7\pi i}{9}}$  & $e^{-\frac{8\pi i}{9}}+3e^{\frac{\pi i}{9}}$  & $3e^{-\frac{5\pi i}{9}}+e^{\frac{4\pi i}{9}}$ \tabularnewline
\hline
10  & $2e^{\frac{2\pi i}{9}}+2e^{\frac{8\pi i}{9}}$  & $e^{-\frac{8\pi i}{9}}+3e^{\frac{\pi i}{9}}$  & $3e^{-\frac{5\pi i}{9}}+e^{\frac{4\pi i}{9}}$  & $e^{-\frac{2\pi i}{9}}+3e^{\frac{7\pi i}{9}}$ \tabularnewline
\hline
11  & $2e^{-\frac{\pi i}{9}}+2e^{\frac{5\pi i}{9}}$  & $3e^{-\frac{5\pi i}{9}}+e^{\frac{4\pi i}{9}}$  & $e^{-\frac{2\pi i}{9}}+3e^{\frac{7\pi i}{9}}$  & $e^{-\frac{8\pi i}{9}}+3e^{\frac{\pi i}{9}}$ \tabularnewline
\hline
12  & $2e^{-\frac{7\pi i}{9}}+2e^{\frac{5\pi i}{9}}$  & $e^{-\frac{2\pi i}{9}}+3e^{\frac{7\pi i}{9}}$  & $e^{-\frac{8\pi i}{9}}+3e^{\frac{\pi i}{9}}$  & $3e^{-\frac{5\pi i}{9}}+e^{\frac{4\pi i}{9}}$ \tabularnewline
\hline
13  & $2e^{-\frac{\pi i}{9}}+2e^{-\frac{7\pi i}{9}}$  & $e^{-\frac{8\pi i}{9}}+3e^{\frac{\pi i}{9}}$  & $3e^{-\frac{5\pi i}{9}}+e^{\frac{4\pi i}{9}}$  & $e^{-\frac{2\pi i}{9}}+3e^{\frac{7\pi i}{9}}$ \tabularnewline
\hline
14  & $e^{-\frac{8\pi i}{9}}+3e^{\frac{\pi i}{9}}$  & $2e^{-\frac{4\pi i}{9}}+2e^{\frac{2\pi i}{9}}$  & $2e^{\frac{2\pi i}{9}}+2e^{\frac{8\pi i}{9}}$  & $2e^{-\frac{4\pi i}{9}}+2e^{\frac{8\pi i}{9}}$ \tabularnewline
\hline
15  & $3e^{-\frac{5\pi i}{9}}+e^{\frac{4\pi i}{9}}$  & $2e^{\frac{2\pi i}{9}}+2e^{\frac{8\pi i}{9}}$  & $2e^{-\frac{4\pi i}{9}}+2e^{\frac{8\pi i}{9}}$  & $2e^{-\frac{4\pi i}{9}}+2e^{\frac{2\pi i}{9}}$ \tabularnewline
\hline
16  & $e^{-\frac{2\pi i}{9}}+3e^{\frac{7\pi i}{9}}$  & $2e^{-\frac{4\pi i}{9}}+2e^{\frac{8\pi i}{9}}$  & $2e^{-\frac{4\pi i}{9}}+2e^{\frac{2\pi i}{9}}$  & $2e^{\frac{2\pi i}{9}}+2e^{\frac{8\pi i}{9}}$ \tabularnewline
\hline
17  & $3e^{-\frac{8\pi i}{9}}+e^{\frac{\pi i}{9}}$  & $2e^{-\frac{4\pi i}{9}}+2e^{\frac{2\pi i}{9}}$  & $2e^{\frac{2\pi i}{9}}+2e^{\frac{8\pi i}{9}}$  & $2e^{-\frac{4\pi i}{9}}+2e^{\frac{8\pi i}{9}}$ \tabularnewline
\hline
18  & $e^{-\frac{5\pi i}{9}}+3e^{\frac{4\pi i}{9}}$  & $2e^{\frac{2\pi i}{9}}+2e^{\frac{8\pi i}{9}}$  & $2e^{-\frac{4\pi i}{9}}+2e^{\frac{8\pi i}{9}}$  & $2e^{-\frac{4\pi i}{9}}+2e^{\frac{2\pi i}{9}}$ \tabularnewline
\hline
19  & $3e^{-\frac{2\pi i}{9}}+e^{\frac{7\pi i}{9}}$  & $2e^{-\frac{4\pi i}{9}}+2e^{\frac{8\pi i}{9}}$  & $2e^{-\frac{4\pi i}{9}}+2e^{\frac{2\pi i}{9}}$  & $2e^{\frac{2\pi i}{9}}+2e^{\frac{8\pi i}{9}}$ \tabularnewline
\hline
\end{tabular}
\par\end{center}

\begin{center}
\begin{tabular}{|c|c|c|c|}
\hline
 & 17  & 18  & 19\tabularnewline
\hline
\hline
0  & $2$  & $2$  & $2$\tabularnewline
\hline
1  & $2e^{-\frac{2\pi i}{3}}$  & $2e^{-\frac{2\pi i}{3}}$  & $2e^{-\frac{2\pi i}{3}}$\tabularnewline
\hline
2  & $2e^{\frac{2\pi i}{3}}$  & $2e^{\frac{2\pi i}{3}}$  & $2e^{\frac{2\pi i}{3}}$\tabularnewline
\hline
3  & $-2$  & $-2$  & $-2$\tabularnewline
\hline
4  & $2e^{\frac{\pi i}{3}}$  & $2e^{\frac{\pi i}{3}}$  & $2e^{\frac{\pi i}{3}}$\tabularnewline
\hline
5  & $2e^{-\frac{\pi i}{3}}$  & $2e^{-\frac{\pi i}{3}}$  & $2e^{-\frac{\pi i}{3}}$\tabularnewline
\hline
6  & $0$  & $0$  & $0$\tabularnewline
\hline
7  & $0$  & $0$  & $0$\tabularnewline
\hline
8  & $3e^{-\frac{5\pi i}{9}}+e^{\frac{4\pi i}{9}}$  & $e^{-\frac{2\pi i}{9}}+3e^{\frac{7\pi i}{9}}$  & $e^{-\frac{8\pi i}{9}}+3e^{\frac{\pi i}{9}}$\tabularnewline
\hline
9  & $e^{-\frac{2\pi i}{9}}+3e^{\frac{7\pi i}{9}}$  & $e^{-\frac{8\pi i}{9}}+3e^{\frac{\pi i}{9}}$  & $3e^{-\frac{5\pi i}{9}}+e^{\frac{4\pi i}{9}}$\tabularnewline
\hline
10  & $e^{-\frac{8\pi i}{9}}+3e^{\frac{\pi i}{9}}$  & $3e^{-\frac{5\pi i}{9}}+e^{\frac{4\pi i}{9}}$  & $e^{-\frac{2\pi i}{9}}+3e^{\frac{7\pi i}{9}}$\tabularnewline
\hline
11  & $3e^{-\frac{5\pi i}{9}}+e^{\frac{4\pi i}{9}}$  & $e^{-\frac{2\pi i}{9}}+3e^{\frac{7\pi i}{9}}$  & $e^{-\frac{8\pi i}{9}}+3e^{\frac{\pi i}{9}}$\tabularnewline
\hline
12  & $3e^{-\frac{2\pi i}{9}}+e^{\frac{7\pi i}{9}}$  & $3e^{-\frac{8\pi i}{9}}+e^{\frac{\pi i}{9}}$  & $e^{-\frac{5\pi i}{9}}+3e^{\frac{4\pi i}{9}}$\tabularnewline
\hline
13  & $3e^{-\frac{8\pi i}{9}}+e^{\frac{\pi i}{9}}$  & $e^{-\frac{5\pi i}{9}}+3e^{\frac{4\pi i}{9}}$  & $3e^{-\frac{2\pi i}{9}}+e^{\frac{7\pi i}{9}}$\tabularnewline
\hline
14  & $2e^{-\frac{4\pi i}{9}}+2e^{\frac{2\pi i}{9}}$  & $2e^{\frac{2\pi i}{9}}+2e^{\frac{8\pi i}{9}}$  & $2e^{-\frac{4\pi i}{9}}+2e^{\frac{8\pi i}{9}}$ \tabularnewline
\hline
15  & $2e^{\frac{2\pi i}{9}}+2e^{\frac{8\pi i}{9}}$  & $2e^{-\frac{4\pi i}{9}}+2e^{\frac{8\pi i}{9}}$  & $2e^{-\frac{4\pi i}{9}}+2e^{\frac{2\pi i}{9}}$\tabularnewline
\hline
16  & $2e^{-\frac{4\pi i}{9}}+2e^{\frac{8\pi i}{9}}$  & $2e^{-\frac{4\pi i}{9}}+2e^{\frac{2\pi i}{9}}$  & $2e^{\frac{2\pi i}{9}}+2e^{\frac{8\pi i}{9}}$\tabularnewline
\hline
17  & $2e^{-\frac{7\pi i}{9}}+2e^{\frac{5\pi i}{9}}$  & $2e^{-\frac{\pi i}{9}}+2e^{-\frac{7\pi i}{9}}$  & $2e^{-\frac{\pi i}{9}}+2e^{\frac{5\pi i}{9}}$\tabularnewline
\hline
18  & $2e^{-\frac{\pi i}{9}}+2e^{-\frac{7\pi i}{9}}$  & $2e^{-\frac{\pi i}{9}}+2e^{\frac{5\pi i}{9}}$  & $2e^{-\frac{7\pi i}{9}}+2e^{\frac{5\pi i}{9}}$\tabularnewline
\hline
19  & $2e^{-\frac{\pi i}{9}}+2e^{\frac{5\pi i}{9}}$  & $2e^{-\frac{7\pi i}{9}}+2e^{\frac{5\pi i}{9}}$  & $2e^{-\frac{\pi i}{9}}+2e^{-\frac{7\pi i}{9}}$\tabularnewline
\hline
\end{tabular}
\par\end{center}

\end{document}